\documentclass{amsart}
\usepackage[utf8]{inputenc}
\usepackage[nobysame,abbrev,alphabetic]{amsrefs}
\usepackage{amsmath}
\usepackage{amssymb}
\usepackage{xcolor}
\usepackage{constants}
\usepackage{enumitem} 
\usepackage{hyperref}
\usepackage{mathtools}
\usepackage{changepage}
\newtheorem{thm}{Theorem}[section]
\newtheorem{lem}[thm]{Lemma}

\newtheorem{cor}[thm]{Corollary}
\newtheorem{defn}[thm]{Definition}
\newtheorem{remark}{Remark}
\newcommand{\abs}[1]{\left\vert#1\right\vert}
\newcommand{\R}{\mathbb{R}}

\newcommand{\Sp}{\mathbb{S}}
\newcommand{\set}[1]{\left\{#1\right\}}

\DeclareMathOperator{\diam}{Diam}

\newcommand{\Sc}{\mathcal{S}} 

\title{On the Stability of Llarull's Theorem in Dimension Three}

\author{Brian Allen}
\address[Brian Allen]{Lehman College, CUNY}
\email{brianallenmath@gmail.com}

\author{Edward Bryden$^\dagger$}
\address[Edward Bryden]{Universiteit Antwerpen}
\email{etbryden@gmail.com}
\thanks{$^\dagger$funded by FWO grant 12F0223N}

\author{Demetre Kazaras}
\address[Demetre Kazaras]{Duke University}
\email{demetre.kazaras@duke.edu}

\begin{document}

\maketitle
    
\begin{abstract}
Llarull's Theorem states that any Riemannian metric on the $n$-sphere which has scalar curv{\-}ature greater than or equal to $n(n-1)$, and whose distance function is bounded below by the unit sphere's, is isometric to the unit sphere.  Gromov later posed the {\emph{Spherical Stability Problem}}, which probes the flexibility of this fact. We give a resolution to this problem in dimension $3$. Informally, the main result asserts that a sequence of Riemannian $3$-spheres whose distance functions are bounded below by the unit sphere's with uniformly bounded Cheeger isoperimetric constant and scalar curvatures tending to $6$ must approach the round $3$-sphere in the volume preserving Sormani-Wenger Intrinsic Flat sense. The argument is based on a proof of Llarull's Theorem due to Hirsch-Kazaras-Khuri-Zhang using spacetime harmonic functions.
\end{abstract}
    
\section{Introduction}

In a pair of influential works \cites{GL1,GL2} Gromov-Lawson developed a rich theory for Riemannian manifolds satisfying a lower scalar curvature bound. Llarull studied an extreme instance of one of these ideas, and proved a remarkable characterization of the round sphere \cite{Llarull}, now known as Llarull's Theorem. The fundamental consequence of this theorem is the following comparison principle and rigidity statement: 
\vspace{.1in}
{\begin{adjustwidth}{30pt}{30pt}
    {\em {\small{Any change to a round sphere's Riemannian structure which increases
its scalar curvature must decrease some distances. 
Moreover, if a Riemannian $n$-sphere has scalar curvature at least $n(n-1)$ and its distance function is bounded below by the unit sphere's, then it is isometric to the unit sphere.}}}
\end{adjustwidth}}
\vspace{.1in}
This fact belongs to a varied collection of rigidity statements satisfied by Riemann{\-}ian manifolds with lower scalar curvature bounds. Aiming to further probe the structure of such manifolds, there is a developing interest in understanding the extent to which these rigid properties are {\em{stable}}. The survey articles \cites{GromovFour, SormaniIAS} describe details on this program and relevant conjectures. We apply this line of questioning to Llarull's Theorem: is a Riemannian $n$-sphere with larger distances compared to the unit sphere and scalar curvature almost $n(n-1)$ close to the unit sphere in some geometric sense? This was named the {\emph{Spherical Stability Problem}} by Gromov \cites{GromovBoundaries}. The main result of the present work, Theorem \ref{thm-Main}, provides a resolution in dimension $3$.

In considering this problem, one encounters the phenomena of {\emph{splines}} and the issue of {\emph{other worlds}}, also known as {\emph{bubbles}} or {\emph{bags of gold}}. From ideas dating back to Gromov-Lawson \cite{GL1} and Schoen-Yau \cite{Schoen-Yau-Structure}, one expects examples of manifolds nearly satisfying the hypotheses of Llarull's rigidity theorem to have significant regions resembling the product of a small $2$-sphere with an interval (splines) and regions of almost completely uncontrolled geometry (other worlds) hidden behind a small neck-like region, see Figures \ref{fig:cheegerpic} and \ref{fig-art}. Rigorous examples of this type have been constructed by Sweeney \cite{Sweeney}. With this in mind, we work with a topology on the space of Riemannian $3$-spheres which is insensitive to splines and consider a class of geometries which does not support extreme bubbling phenomena, but is not unnecessarily constrained. The topology we consider is generated by the {\emph{Sormani-Wenger Intrinsic Flat distance}} \cite{Sormani-Wenger}, discussed below Theorem \ref{thm-Main}. First, we specify the class of $3$-spheres considered.
 
\begin{defn}\label{defn:FamilyOfMetrics}
        Given $V,D,\overline{m},\Lambda>0$, let $\Sc(V,D,\overline{m},\Lambda)$ denote the class of Riemannian $3$-spheres $(\Sp^3,g)$ which satisfy a
        \begin{enumerate}[label=(\roman*)]
            \item distance lower bound: $g\geq g_{\Sp^3}$,
            \item volume upper bound: $\abs{\Sp^3}_{g}\leq V$,
            \item diameter bound: $\diam(\Sp^3,g)\le D$,
            \item scalar curvature bound: $\left\|\left(6-R_{g}\right)^{+}\right\|_{L^{2}(g)}^{1/2}\leq \overline{m}$,
            \item Cheeger constant lower bound: $IN_{1}(g)\geq\Lambda$.
        \end{enumerate}
        In the above, $R_g$ denotes the scalar curvature of $(\Sp^3,g)$ and $IN_1(g)$ denotes the Cheeger constant given by $\inf_{\Omega\subset\Sp^3}\tfrac{|\partial\Omega|_g}{\min(|\Omega|_g,|\Sp^3\setminus\Omega|_g)}$.
    \end{defn}
    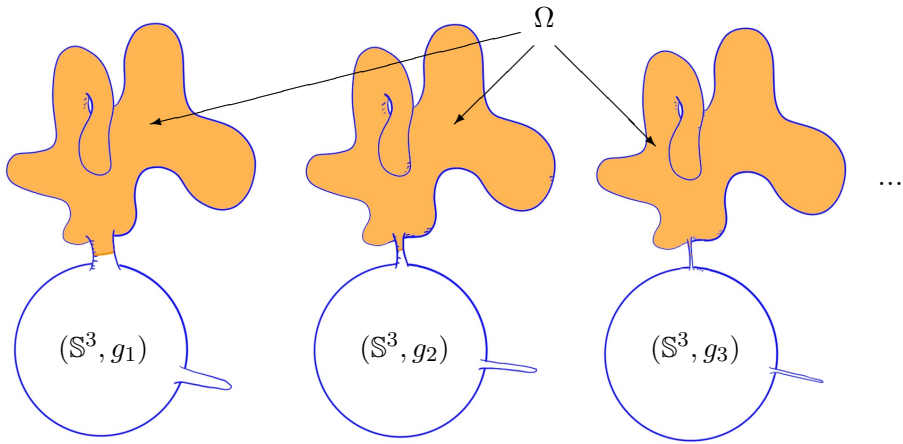
\begin{figure}[h]
    \centering
    \includegraphics[scale=.2]{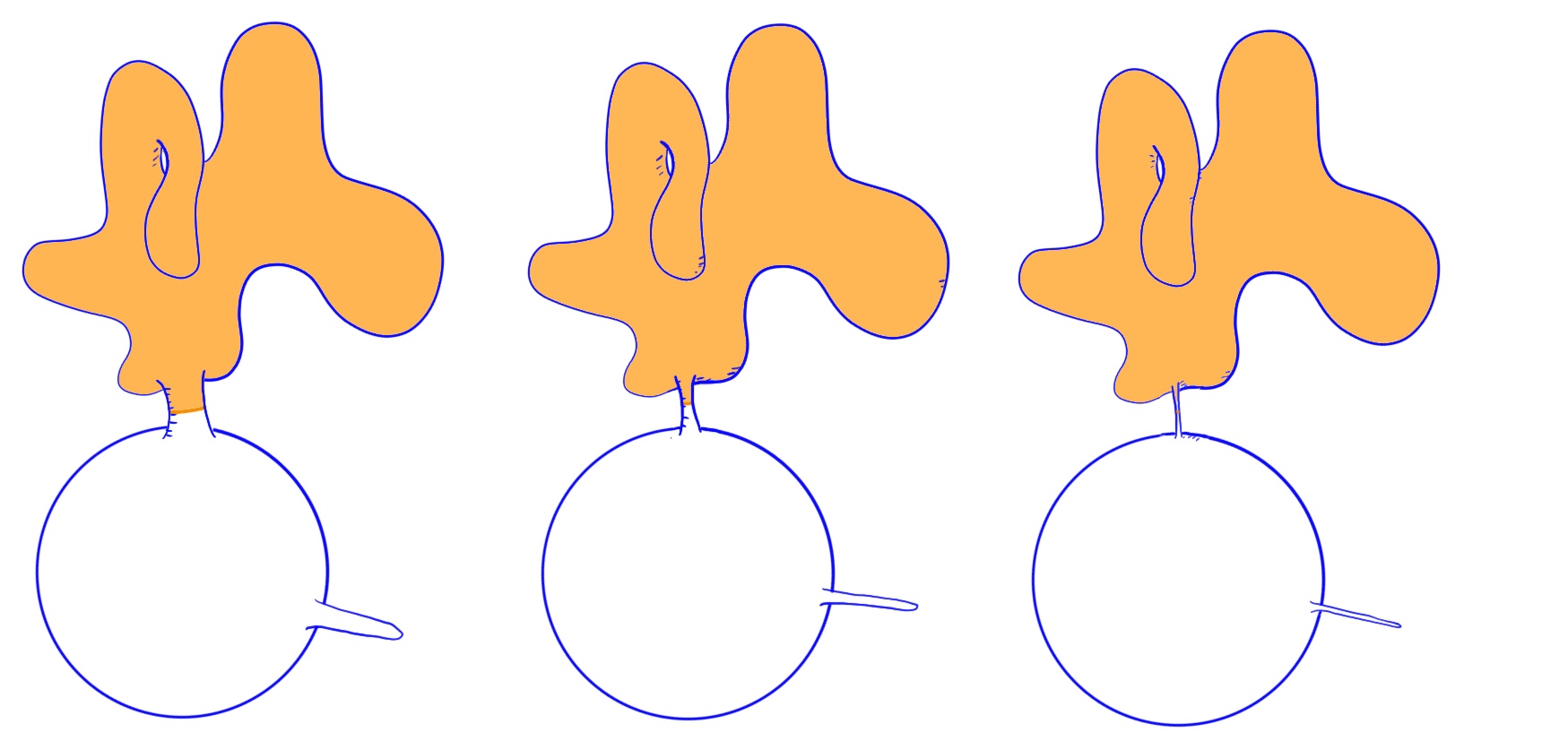}
       \begin{picture}(0,0)
\put(-20,100){{\Large{...}}}
\put(-150,160){{\large{$\Omega$}}}
\put(-155,158){\vector(-4,-1){140}}
\put(-150,153){\vector(-1,-1){30}}
\put(-142,153){\vector(1,-1){38}}
\put(-330,35){{\large{$(\Sp^3,g_1)$}}}
\put(-216,35){{\large{$(\Sp^3,g_2)$}}}
\put(-106,35){{\large{$(\Sp^3,g_3)$}}}
\end{picture}
    \caption{A sequence failing the Cheeger constant lower bound, since the orange region $\Omega$ has large volume with a shrinking boundary. Conditions $(i)$-$(iii)$ and smallness of $(6-R_g)^+$ force the smallness of necks and splines, but have almost no impact on the other worlds. }
    \label{fig:cheegerpic}
\end{figure}
We describe how the above conditions shape the class of geometries. First, item $(i)$ is equivalent to a distance lower bound. This and scalar curvature bounds are essential requirements for the Spherical Stability Problem. The volume and diameter restrictions allow us to avoid technical non-compactness issues, and are standard in geometric stabil{\-}ity questions, see \cites{GromovFour, SormaniIAS}. Finally, the Cheeger constant lower bound addresses the issue of other worlds. There exists sequences satisfying $(i)-(iv)$, $R_{g_i}\to6$, and with degenerating Cheeger constants which fail to converge to the unit sphere in the Intrinsic Flat sense, see Figure \ref{fig:cheegerpic}. We emphasize that this last condition is mild in the sense that it does not prohibit the main geometrical challenges of the stability problem, namely the presence of splines. With this class of manifolds defined, we are now able to state the main theorem.
    \begin{thm}\label{thm-Main}
        Let $V,D,\overline{m},\Lambda>0$. If a sequence $\{(\Sp^3,g_i)\}_{i=1}^\infty$ of Riemannian $3$-spheres in $\Sc(V,D,\overline{m},\Lambda)$ satisfies
        \begin{align}
         \left\|\left(6-R_{g_i}\right)^{+}\right\|_{L^{2}(g_i)}^{1/2} \rightarrow 0,
        \end{align}
        then it converges in the volume preserving Sormani-Wenger Intrinsic Flat sense to the unit $3$-sphere:
        \begin{align}
            d_{\mathcal{VF}}((\Sp^3,g_i),(\Sp^3,g_{\mathbb{S}^3}))\rightarrow 0.
        \end{align}
    \end{thm}

\begin{figure}[h] 
   \center{\includegraphics[width=.38\textwidth]{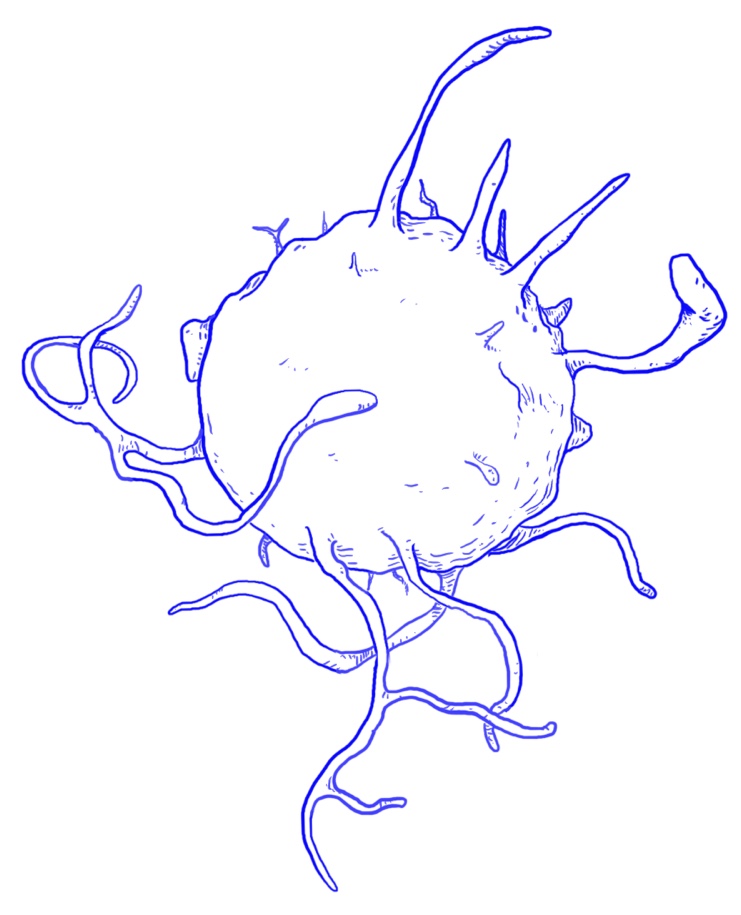}}
\caption{An artistic depiction of a Riemannian $3$-sphere lying far out in a tail of the sequence $(\Sp^3,g_i)$ appearing in Theorem \ref{thm-Main}. The shown tendrils enclose, in an intrinsic sense, little volume and have a negligible impact on its $d_{\mathcal{VF}}$-distance to the round sphere.}
   \label{fig-art}
\end{figure}

We remark that a precise formulation of the Spherical Stability Problem was communicated by Sormani in the work of Sweeney \cite{Sweeney}. In dimension $3$, their conjecture is stated for a class of manifolds similar to the one given in Definition \ref{defn:FamilyOfMetrics}, except that the Cheeger constant lower bound is replaced by a lower bound on the area of minimal surfaces, called a {\emph{MinA}} condition. The MinA condition was first suggested by Sormani \cite[Remark 9.10]{SormaniIAS} as a means for addressing the issue of other worlds by detecting the narrow neck in their boundary. Theorem \ref{thm-Main} suggests that a Cheeger constant lower bound is a possible alternative to consider in other scalar curvature stability problems. A Cheeger constant lower bound may be preferable to the MinA condition due to the fact that $IN_1(g)$ is continuous in, for instance, the $C^0$-topology whereas the smallest minimal surface area is not. In particular, the Cheeger constant allows for arbitrarily small other worlds, whereas the MinA condition does not.

The Sormani-Wenger Intrinsic Flat distance appearing in Theorem \ref{thm-Main} is a geometric measurement of the difference between Riemannian manifolds which is particularly suited for problems involving lower scalar curvature bounds. To describe the prototypical situation, one can construct sequences of Riemannian $3$-spheres with increasingly many splines, bounded volume, and positive scalar curvature. Such a sequence possesses an unbounded number of disjoint metric balls of fixed radius. Due to Gromov's compactness theorem, these examples cannot have a Gromov-Hausdorff limit. However, such sequences are known to converge with respect to the Sormani-Wenger Intrinsic Flat distance. To demonstrate this convergence in Theorem \ref{thm-Main}, we take advantage of the Volume Above Distance Below Theorem of Allen-Perales-Sormani \cite{Allen-Perales-Sormani-VADB}, reviewed in Section \ref{subsec:VADB}. This theorem asserts that a sequence of Riemannian manifolds $\{(M,g_i)\}_{i=1}^\infty$ with bounded diameter, volumes converging to that of a Riemannian manifold $(M,g)$, and distance bounds from below $g_i \ge g$ will converge to $(M,g)$ in the volume preserving Sormani-Wenger Intrinsic Flat sense. In the setting of Llarull's Theorem, we are given that each metric under consideration is bounded from below by the round metric on the three-sphere. Thus, our objective is to establish volume convergence. 

The proof of Theorem \ref{thm-Main} centers on the analysis of spacetime harmonic functions introduced by Hirsch-Kazaras-Khuri in \cite{hirsch-kazaras-khuri}, where an integral identity based on fundamental work of Stern \cite{Stern} was used to study mass in general relativity. These ideas were used by Hirsch-Kazaras-Khuri-Zhang in \cite{Hirsch-Kazaras-Khuri-Zhang} to give a new proof of Llarull's Theorem in dimension $3$. In this particular setting, we call the relevant spacetime harmonic functions {\em{Llarull potentials}}. Llarull potentials solve a quasilinear elliptic equation with coefficients which become singular near the poles, and satisfy an important integral formula relating scalar curvature to $C^2$ information of the Llarull potential. The quantitative nature of this formula is an essential tool for showing stability. 

To demonstrate the necessary volume convergence, we identify regions in $(\Sp^3,g_i)$ with strong pointwise control on its Llarull potential. Leveraging the integral formula mentioned above, we show these regions have volume forms comparable to the round sphere's. The fundamental step is an in-depth analysis of the PDE near its singular poles, showing the well-controlled regions cannot collapse along the sequence. Next, the Cheeger constant lower bound is combined with a strategy used by Dong \cite{Dong-PMT_Stability} and Dong-Song \cite{Dong-Song} to show the well-controlled regions encompass almost all of the sphere far out in the sequence.  

In Section \ref{sect: background}, we review background material on Llarull potentials, Cheeger's constant, the proof of Llarull's Theorem in \cite{Hirsch-Kazaras-Khuri-Zhang}, and the Volume Above Distance Below Theorem. Section \ref{sect: global estimates} contains uniform integral estimates for Llarull potentials. These estimates all follow from standard PDE methods applied to the Llarull potentials and the control  provided by the proof of Llarull's Theorem. Our critical work takes place in Section \ref{sect: polar estimates}, where estimates are established near the poles $\{p,-p\} \in \Sp^3$ for Llarull potentials whose Riemannian metrics are in the class defined in Definition \ref{defn:FamilyOfMetrics}. In Section \ref{sect: VADB convergence}, we give the proof of Theorem \ref{thm-Main} by combining the estimates of the previous section to Theorem \ref{thm-VADB}.

\section{Background}\label{sect: background}
    In this section we review background material and preliminary arguments necessary to understand proofs given in subsequent sections.
        
    {\bf Conventions and notations:} given a point $p\in\Sp^3$, let $\theta_p(x)$ denote the distance function from $p$ to $x$ as measured using the unit round metric. When the point $p$ is clear in context, we simply write $\theta=\theta_p$. We use $B^\Sp(x,r)$ to denote the ball of radius $r$ about $x\in\Sp^3$ in the unit $3$-sphere. Notice that $\theta^{-1}(r)=\partial B^\Sp(p,r)$ and we will use these notations interchangeably. Given a Riemannian metric $g$ and a set $A$, we use $|A|_g$ to denote either the volume or area of $A$, depending on the context, as measured by $g$. Given a tensor $T$, we write $|T|$ for its norm with respect to $g$. Unless mentioned explicitly, we fix the parameters $V,D,\Lambda,\overline{m}>0$ for the remainder of the paper.

\subsection{Llarull potentials on $\Sp^3$}
    In what follows, we gather together and make explicit some results proven in \cite[Section 7]{Hirsch-Kazaras-Khuri-Zhang}.
    \begin{lem}\label{lem:existence}
        Given a Riemannian $3$-sphere $(\Sp^3,g)$ satisfying $g\geq g_{\Sp^3}$ and antipodal points $\pm p$, there exists a function $u\in C^{2,\alpha}(\Sp^3\setminus \{p,-p\})\cap \mathrm{Lip}(\Sp^3)$ so that 
        \begin{enumerate}[label=(\roman*)]
            \item $\Delta_g u+3\cot(\theta)|\nabla u|=0$ weakly across $\Sp^3$ and $u(\pm p)=\pm1$,
            \item  the following integral inequalities hold:
            \begin{equation}\label{eq:csc2_norm_grad_u}
                \int_{\Sp^3}\csc^2(\theta)\abs{\nabla u}dV_{g}\leq8\pi+\frac{1}{2}\int_{\Sp^3}\left(6-R_{g}\right)^{+}\abs{\nabla u}dV_{g},
            \end{equation}
            \begin{equation}\label{eq:csc_grad_u-grad_theta_alignment}
                \int_{\Sp^3}\csc^{2}(\theta)\left(\abs{\nabla u}+g(\nabla u,\nabla\theta)\right)dV_{g}\leq\frac14 \int_{\Sp^3}\left(6-R_{g}\right)^{+}\abs{\nabla u}dV_{g},
            \end{equation}
            \begin{equation}\label{eq:mass_formula}
                \int_{\Sp^3}\frac{|\nabla^2u+\cot(\theta)|\nabla u| g|^2}{|\nabla u|}dV_g \le \int_{\Sp^3}(6-R_g)^{+}|\nabla u| dV_g.
            \end{equation}
        \end{enumerate}
    \end{lem}

\begin{remark}
    We call $u$ in Lemma \ref{lem:existence} a {\emph{Llarull potential centered at $\pm p$}}.
\end{remark}

\begin{proof}
    This result is essentially proven in \cite{Hirsch-Kazaras-Khuri-Zhang}, but we will give an overview for convenience. The function $u$ is constructed as a limit of solutions on compact domains exhausting $\Sp^3\setminus\{p,-p\}$. Namely for $\varepsilon>0$ write $A_\varepsilon$ to denote $\Sp^3\setminus\left(B^\Sp(p,\varepsilon)\cup B^\Sp(-p,\varepsilon)\right)$, and let $u_\varepsilon$ be the solution to the Dirichlet problem with bounded coefficients
    \begin{equation}
    \begin{cases}
        \Delta_g u_\varepsilon+3\cot(\theta)|\nabla u_\varepsilon|=0&\text{ in } A_\varepsilon\\
        u_\varepsilon=\pm1&\text{ on }\partial B^{\Sp}(\pm p,\varepsilon).
    \end{cases}
    \end{equation}
    The existence of $u_\varepsilon$ is given in \cite[Section 4]{hirsch-kazaras-khuri}. The maximum principle and standard interior Schauder estimates give uniform $C^{2,\alpha}$-bounds, $\alpha\in(0,1)$, for $u_\varepsilon$ on compact subsets of $\Sp^3\setminus\{p,-p\}$. This allows one to apply a diagonal argument and pass to a subsequence so that $u_\varepsilon\to u$ in $C^{2,\beta}_{loc}(\Sp^3\setminus\{p,-p\})$ as $\varepsilon\to 0$, for $\beta<\alpha$. In fact, \cite[Lemma A.1]{Hirsch-Kazaras-Khuri-Zhang} shows that $|\nabla u_\varepsilon|$ is uniformly bounded in $\varepsilon$, and that the limiting solution $u$ extends to a globally Lipschitz function across $\pm p$ so that $u(\pm p)=\pm1$. Since $u$ is a strong solution away from $\pm p$ and has a uniform gradient bound, it is a weak solution to the Llarull potential equation over all of $\Sp^3$, and item $(i)$ follows.

    Next, we point out how to establish \eqref{eq:csc2_norm_grad_u}, \eqref{eq:csc_grad_u-grad_theta_alignment}, and \eqref{eq:mass_formula}. Inequality \eqref{eq:mass_formula} is exactly the statement of \cite[Theorem 2.7]{Hirsch-Kazaras-Khuri-Zhang}, though in the proof extra positive terms are discarded. Retention of these terms leads directly to the remaining claimed inequalities. We start by reorganizing \cite[Inequalities (7.1) and (7.2)]{Hirsch-Kazaras-Khuri-Zhang} to find
    \begin{align}\label{eq:massformulaproof}
    \begin{split}
        &\int_{A_\varepsilon}(6-R_g)^+|\nabla u_\varepsilon|dV_g\\
        &\geq \int_{A_\varepsilon}\frac{|\nabla^2u_\varepsilon+\cot(\theta)|\nabla u_\varepsilon|g|^2}{|\nabla u_\varepsilon|}+4\csc^2(\theta)\left(|\nabla u_\varepsilon|+g(\nabla u_\varepsilon,\nabla\theta)\right)dV_g\\
        &+\left(\int_{A_\varepsilon}2\csc^2(\theta)|\nabla u_\varepsilon|dV_g-16\pi\right)+2\int_{\partial A_\varepsilon}|\nabla u_\varepsilon|(H+2\cot(\varepsilon))dA_g
    \end{split}
    \end{align}
    where $H$ denotes the mean curvature of $\partial B^\Sp(\pm p,\varepsilon)$ with respect to normals pointing towards the poles. To take the liminf as $\varepsilon\to0$, we first claim that the boundary term on the right side of \eqref{eq:massformulaproof} vanishes as $\varepsilon\to0$. Indeed, since $(\Sp^3,g)$ is smooth near $\pm p$, one may expand $H=-\tfrac{2}{\varepsilon}+O(1)$ and so the top order term in $H+2\cot(\varepsilon)$ cancels. Then, using the fact that $|\nabla u_\varepsilon|$ is bounded and the area of the $\partial B^\Sp(\pm p,\varepsilon)$ tends to $0$, we find that the boundary integral vanishes as $\varepsilon\to0$. Gearing up to use Fatou's Lemma, note that the second integrand on the right side of \eqref{eq:massformulaproof} is non-negative by combining the Cauchy-Schwartz inequality with the fact $|\nabla \theta|\leq 1$, which follows from the assumption $g\geq g_{\Sp^3}$. Taking the liminf of \eqref{eq:massformulaproof}, applying the Dominated Convergence Theorem to the left side and Fatou's Lemma to the right side, we may conclude
    \begin{align}\label{eq:massformulaproof2}
    \begin{split}
        &\int_{\Sp^3}(6-R_g)^+|\nabla u|dV_g\\
        &\geq \int_{\Sp^3}\frac{|\nabla^2u+\cot(\theta)|\nabla u|g|^2}{|\nabla u|}+4\csc^2(\theta)\left(|\nabla u|+g(\nabla u,\nabla\theta)\right)dV_g\\
        &+\left(\int_{\Sp^3}2\csc^2(\theta)|\nabla u|dV_g-16\pi\right).
    \end{split}
    \end{align}
    To apply Fatou's Lemma to the first term on the right side of \eqref{eq:massformulaproof}, one must apply an argument which uses the strong convergence of $u_\varepsilon$ and the fact that $u$ is a non-constant solution to an elliptic PDE to show pointwise almost everywhere convergence. See the end of the argument in \cite[Theorem 2.7]{Hirsch-Kazaras-Khuri-Zhang} for details.

    Lastly, observe that the final term on the right side of \eqref{eq:massformulaproof2} is non-negative by \cite[Inequality (7.4)]{Hirsch-Kazaras-Khuri-Zhang}. It follows that each of the three grouped terms on the right side of \eqref{eq:massformulaproof2} are individually bounded by the left side, leading to \eqref{eq:csc2_norm_grad_u}, \eqref{eq:csc_grad_u-grad_theta_alignment}, and \eqref{eq:mass_formula}.
\end{proof}

    \subsection{Cheeger Constant}\label{subsec:Cheeger}
In this section we review the definition of the Cheeger constant which is the boundary case of a wider definition of $\alpha$-isoperimetric constants defined in \cite{li_2012}. 

\begin{defn}\cite[Definition 9.2]{li_2012}\label{def:Cheeger}
    Given a closed Riemannian manifold $(M,g)$, define $IN_1(M,g)$ as the infimum
    \begin{equation}
        \inf\left\lbrace\frac{|S|_g}{\min\{|\Omega_1|_g,|\Omega_2|_g\}}:M=\Omega_1\cup S \cup \Omega_2,\partial \Omega_1=S=\partial\Omega_2\right\rbrace.
    \end{equation}
    $IN_1(M,g)$ is called the  Neumann $1-$isoperimetric, or Cheeger, constant of $(M,g)$.
\end{defn}
A positive lower bound on the Cheeger constant rules out dramatic bubble- or dumbbell-like regions since it prohibits surfaces of small area to form between two regions with substantial volume.  We also remind the reader of the relationship between the isoperimetric type constant above and the Sobolev constant defined below.

\begin{defn}\cite[Definition 9.4]{li_2012}\label{def:Sobolev}
     Given a closed Riemannian manifold $(M,g)$, define $SN_1(M,g)$ as the infimum
     \begin{equation}\label{eq:def_of_Neumann_Sobolev_const}
         \inf\left\lbrace\frac{\|\nabla f\|_{L^{1}(g)}}{\inf_{k\in\mathbb{R}}\|f-k\|_{L^1(g)}}:f\in W^{1,1}(g)\right\rbrace.
     \end{equation}
     $SN_1(M,g)$ is called the Neumann $1$-Sobolev constant of $(M,g)$.
\end{defn}
The infimum in \eqref{eq:def_of_Neumann_Sobolev_const} is actually achieved for any $f\in W^{1,1}(g)$. This is because $k\mapsto\int|f-k|dV_g$ is continuous, positive, and goes to $\infty$ as $k\to\pm\infty$. In what follows, whenever we use \eqref{eq:def_of_Neumann_Sobolev_const} to analyze an element of $W^{1,1}(g)$, we will always work with a constant which achieves the infimum.
The following important relationship between Definition \ref{def:Cheeger} and Definition \ref{def:Sobolev} is given in \cite[Theorem 9.6]{li_2012}:
\begin{align}
     IN_1(M,g) =SN_1(M,g).
\end{align}
So by assuming a bound on the Cheeger constant we obtain geometric control which rules out bubbling, and we also receive analytic control of Llarull potentials in the form of a Poincare inequality. We will take advantage of this analytic control in Section \ref{sect: global estimates}.

\subsection{Sormani-Wenger Intrinsic Flat Convergence}\label{subsec:VADB}
In order to show volume preserving Sormani-Wenger Intrinsic Flat convergence we will apply the Volume Above Distance Below theorem of the first named author, R. Perales, and C. Sormani which provides geometric hypotheses which imply the desired convergence for a sequence of Riemannian manifolds.
    
    \begin{thm}\cite[Theorem 1.1]{Allen-Perales-Sormani-VADB} \label{thm-VADB}
        Let $M$ be a connected closed orientable manifold, $g_0$ a smooth Riemannian metric on $M$, and $\{g_i\}_{i=1}^\infty$ a sequence of continuous Riemannian metrics on $M$. If the diameter is bounded
        \begin{align}
            \diam(M,g_i)&\le D,
            \end{align}
            volume converges,
            \begin{align}
              |M|_{g_i} &\rightarrow |M|_{g_0}, 
            \end{align}
            and we have that distances converge from below
            \begin{align}
            g_i &\ge \left(1-\frac{1}{i}\right)g_0,
        \end{align}
        then $(M,g_i)$ converges to $(M,g_0)$ in the volume preserving Sormani-Wenger Intrinsic Flat sense
        \begin{align}
            d_{\mathcal{VF}}((M,g_i),(M,g_0)) \rightarrow 0.
        \end{align}
    \end{thm}

    Our main approach to proving Theorem \ref{thm-Main} will be to show volume convergence using Llarull potentials so that we can apply Theorem \ref{thm-VADB}. To this end, we establish a point-picking result which will help control the volume of balls around antipodal points later on.

    \begin{lem}\label{lem:small_theta_balls}
        Fix $V>0$. There is a constant $\Cl{const:radius}(V)$ so that the following holds: Given any sufficiently small radius $r>0$ and Riemannian $3$-sphere $(\Sp^3,g)$ with volume $|\Sp^3|_g\leq V$, there are antipodal points $\pm p$ so that 
        \begin{equation}\label{eq:volumeantipodal}
            |B^{\Sp}(p,r)|_g+|B^{\Sp}(-p,r)|_g<\Cr{const:radius}(V)r^3.
        \end{equation}
    \end{lem}
    \begin{proof}
        Fix $r>0$ less than $\tfrac{1}{10^4}$. To use a very coarse estimate, one may find $N:=\left(\tfrac{1}{10^4r}\right)^3$ points $\{p_i\}_{i=1}^N$ in the upper hemisphere so that the constant curvature balls $\{B^{\Sp}(p_i,r),B^{\Sp}(-p_i,r)\}_{i=1}^N$ are all disjoint. It follows that
        \begin{equation}
            V\geq\sum_{i=1}^N\left(\abs{B^{\Sp}(p_i,r)}_{g}+\abs{B^{\Sp}(-p_i,r)}_{g}\right).
        \end{equation}
        Consequently, at least one pair of points $\pm p_{i_0}$ in our list must satisfy the bound $\abs{B^{\Sp}(p_{i_0},r)}_{g}+\abs{B^{\Sp}(-p_{i_0},r)}_{g}\leq10^{12}r^3V$. 
    \end{proof}

    \section{Global Estimates} \label{sect: global estimates}
    In this section we combine the analytic inequalities in Lemma \ref{lem:existence} with the geometric properties satisfied by members of $\mathcal{S}(V,D,\overline{m},\Lambda)$ to deduce a variety of uniform and effective estimates on the Llarull potentials introduced in Lemma \ref{lem:existence} above. We start by establishing integral gradient bounds.
    \begin{lem}\label{lem:L2Grad_u_bound}
        There exists a constant $\Cl{const:L2Grad_u_bound}(\overline{m})$ so that for any $(\Sp^3,g)$ in $\Sc(V,D,\overline{m},\Lambda)$ with Llarull potential $u$ centered at $\pm p$, we have
        \begin{equation}
            \|\nabla u\|_{L^{2}(g)}\leq \Cr{const:L2Grad_u_bound}(\overline{m}).
        \end{equation}
    \end{lem}
    \begin{proof}
        Multiplying both sides of the Llarull potential equation appearing in Lemma \ref{lem:existence} by $u$ and integrating by parts shows
        \begin{equation}
            -\int_{\Sp^{3}}\abs{\nabla u}^2dV_{g}=-3\int_{\Sp^3}\cot({\theta})u\abs{\nabla u}dV_{g}.
        \end{equation}
        According to the maximum principle, $\abs{u}\leq 1$. This allows us to take the absolute value of both sides of the above equation and show
        \begin{align}\label{eq:l2gradbound1}
        \begin{split}
            \int_{\Sp^3}\abs{\nabla u}^2dV_{g}&\leq3\int_{\Sp^3}\abs{\cot({\theta})}\abs{u}\abs{\nabla u}dV_{g}\\
            {}&\leq 3\int_{\Sp^3}\abs{\cot(\theta)}\abs{\nabla u}dV_{g}\\
            {}&\leq3\int_{\Sp^2}\csc^{2}(\theta)\abs{\nabla u}dV_{g},
        \end{split}
        \end{align}
        where the last line follows from the fact $\abs{\cot(\theta)}\leq\csc^2{(\theta)}$.
        Combining \eqref{eq:l2gradbound1} with \eqref{eq:csc2_norm_grad_u} yields
        \begin{equation}
            \int_{\Sp^3}\abs{\nabla u}^2dV_{g}\leq24\pi+\frac{3}{2}\int_{\Sp^3}\left(6-R_{g}\right)^{+}\abs{\nabla u}dV_{g}.
        \end{equation}
        Finally, applying H{\"o}lder's inequality to the integral on the right implies 
        \begin{equation}
            \|\nabla u\|^2_{L^{2}(g)}\leq24\pi+\frac32\|\left(6-R_g\right)^{+}\|_{L^2(g)}\|\nabla u\|_{L^2(g)}.
        \end{equation}
        It is a straightforward dichotomy that either $\|\nabla u\|_{L^2(g)}\leq 1$ or $\|\nabla u\|_{L^{2}(g)}>1$. In the second case, we may divide both sides of the above inequality by $\|\nabla u\|_{L^{2}(g)}$ to obtain
        \begin{equation}
            \|\nabla u\|_{L^{2}(g)}\leq24\pi+\frac32\|\left(6-R_g\right)^{+}\|_{L^{2}(g)}.
        \end{equation}
        In either case, we have demonstrated that
        \begin{equation}
            \|\nabla u\|_{L^2(g)}\leq 24\pi+\frac32 \overline{m}.
        \end{equation}
    \end{proof}
    As a simple consequence, we observe that the $L^2$ gradient bound of Lemma \ref{lem:L2Grad_u_bound} implies an $L^1$ bound.
    \begin{cor}\label{cor:L1_norm_grad_u}
        There exists a constant $\Cl{const:L1_norm_grad_u}(\overline{m})$ so that for any $(\Sp^3,g)$ in $\Sc(V,D,\overline{m},\Lambda)$ with Llarull potential $u$ centered at $\pm p$, we have
        \begin{equation}
            \|\nabla u\|_{L^{1}(g)}\leq\Cr{const:L1_norm_grad_u}(\overline{m}).
        \end{equation}
    \end{cor}
    \begin{proof}
        Since $\csc(\theta)\geq 1$, inequality \eqref{eq:csc2_norm_grad_u} gives
        \begin{equation}
            \int_{\Sp^3}\abs{\nabla u}dV_{g}\leq 8\pi+\frac{1}{2}\int_{\Sp^3}\left(6-R_{g}\right)^{+}\abs{\nabla u}dV_{g}.
        \end{equation}
        Applying H{\"o}lder's inequality and Lemma \ref{lem:L2Grad_u_bound} gives the result.
    \end{proof}
    Given antipodal points $\pm p$ and a function $f$ on $(\Sp^3,g)$, let $\overline{\nabla}^2f$ denote the {\emph{spacetime Hessian} of $f$}, defined by
        \begin{equation}
            \overline{\nabla}^2f:=\nabla^2f+\cot(\theta)\abs{\nabla f}g.
        \end{equation}
    We have the following consequence of \eqref{eq:mass_formula}.
    \begin{lem}\label{lem:u-theta_alignment}
        There exists a constant $\Cl{const:u-theta_alignment}(\overline{m})$ so that for any $(\Sp^3,g)$ in $\Sc(V,D,\overline{m},\Lambda)$ with Llarull potential $u$ centered at $\pm p$, we have
        \begin{align}
            &\int_{\Sp^3}\csc^2(\theta)\abs{\nabla u+\abs{\nabla u}\nabla\theta}dV_{g}\leq\Cr{const:u-theta_alignment}(\overline{m})\left\|\left(6-R_{g}\right)^{+}\right\|^{\frac12}_{L^{2}(g)}\label{eq:u-theta1}
            \\
            &\int_{\Sp^3}\csc({\theta})\abs{\overline{\nabla}^2u}dV_{g}\leq\Cr{const:u-theta_alignment}(\overline{m})\left\|\left(6-R_{g}\right)^{+}\right\|^{\frac12}_{L^{2}(g)}.\label{eq:u-theta2}
        \end{align}
    \end{lem}
    \begin{proof}
        Begin by expanding $\abs{\nabla u+\abs{\nabla u}\nabla\theta}^2$ and applying $|\nabla \theta|\leq 1$ to find
        \begin{align}\label{eq:Grad_u-Grad_theta}
        \begin{split}
            \abs{\nabla u+\abs{\nabla u}\nabla\theta}^2&=\abs{\nabla u}^2+2\abs{\nabla u}g\left(\nabla u,\nabla\theta\right)+\abs{\nabla u}^2\abs{\nabla\theta}^2\\
            {}&\leq 2\abs{\nabla u}\left(\abs{\nabla u}+g\left(\nabla u,\nabla\theta\right)\right).
        \end{split}
        \end{align}
        Taking the square root of the above, multiplying both sides by $\csc^{2}(\theta)$, and then integrating gives us
        \begin{align}\begin{split}
            &\int_{\Sp^3}\csc^2(\theta)\abs{\nabla u+\abs{\nabla u}\nabla\theta}dV_{g}\\
            &\leq\int_{\Sp^3}\left(\csc(\theta)\abs{\nabla u}^{\frac12}\right)\left(\csc(\theta)\sqrt{2\left(\abs{\nabla u}+g\left(\nabla u,\nabla\theta\right)\right)}\right)dV_{g}.
            \end{split}
        \end{align}
        We may apply H{\"o}lder's inequality to the right hand side and use inequality \eqref{eq:csc2_norm_grad_u} in Lemma \ref{lem:existence} in order to obtain
        \begin{align}\begin{split}
            &\int_{\Sp^3}\csc^2(\theta)\abs{\nabla u+\abs{\nabla u}\nabla\theta}dV_{g}\\
            &\leq\left(8\pi+\frac12 \int_{\Sp^3}\left(6-R_{g}\right)^{+}\abs{\nabla u}dV_{g}\right)^{\frac{1}{2}}
            \\
            &\times\left(2\int_{\Sp^3}\csc^2(\theta)\left(\abs{\nabla u}+g\left(\nabla u,\nabla\theta\right)\right)dV_{g}\right)^{\frac12}.
            \end{split}
        \end{align}
        Using Lemma \ref{lem:L2Grad_u_bound} and inequality \eqref{eq:csc_grad_u-grad_theta_alignment}, the above becomes
        \begin{align}\begin{split}
            \int_{\Sp^3}\csc^2(\theta)\abs{\nabla u+\abs{\nabla u}\nabla\theta}dV_{g}\leq&\left(8\pi+\frac{1}{2}\overline{m}\Cr{const:L2Grad_u_bound}(\overline{m})\right)^{\frac12}
            \\
            &\times\left(\frac12\int_{\Sp^3}\left(6-R_{g}\right)^{+}\abs{\nabla u}dV_{g}\right)^{\frac12}.
            \end{split}
        \end{align}
        Leveraging Lemma \ref{lem:L2Grad_u_bound} and H{\"o}lder's inequality once more gives the following 
        \begin{align}\begin{split}
            &\int_{\Sp^3}\csc^2(\theta)\abs{\nabla u+\abs{\nabla u}\nabla\theta}dV_{g}\\
            &\leq\left(8\pi+\frac{1}{2}\overline{m}\Cr{const:L2Grad_u_bound}(\overline{m})\right)^{\frac12}\left(\frac{1}{2}\Cr{const:L2Grad_u_bound}(\overline{m})\|(6-R_{g})^{+}\|_{L^{2}(g)}\right)^{\frac12},
            \end{split}
        \end{align}
        establishing inequality \eqref{eq:u-theta1}.

        To prove \eqref{eq:u-theta2}, we multiply and divide by $\abs{\nabla u}^{\frac{1}{2}}$ and then use 
        H{\"o}lder's inequality:
        \begin{align}
            \int_{\Sp^3}\csc(\theta)\abs{\overline{\nabla}^2u}dV_{g}\leq&\left(\int_{\Sp^3}\csc^{2}(\theta)\abs{\nabla u}dV_{g}\right)^{\frac12}
            \left(\int_{\Sp^{3}}\frac{\abs{\overline{\nabla}^2u}^2}{\abs{\nabla u}}dV_{g}\right)^{\frac12}.
        \end{align}
        The proof is finished by applying inequalities \eqref{eq:csc2_norm_grad_u} and \eqref{eq:mass_formula}.
    \end{proof}

    Now we observe a reformulation of \eqref{eq:mass_formula} which is useful for comparing the gradient of $u$ to the gradient of the Llarull potential on the round sphere, namely $\sin(\theta)$.
    \begin{lem}\label{lem:GradRatio_u-sin_bound}
        There exists a constant $\Cl{const:GradRatio_u-sin_bound}(\overline{m})$ so that for any $(\Sp^3,g)$ in $\Sc(V,D,\overline{m},\Lambda)$ with Llarull potential $u$ centered at $\pm p$, we have
        \begin{equation}
            \int_{\Sp^3}\abs{\nabla\frac{\abs{\nabla u}}{\sin(\theta)}}dV_{g}\leq \Cr{const:GradRatio_u-sin_bound}(\overline{m})\left\|\left(6-R_{g}\right)^{+}\right\|^{\frac12}_{L^{2}(g)}.
        \end{equation}
    \end{lem}
    \begin{proof}
        A calculation gives the following equality:
        \begin{equation}
            \nabla\left(\frac{\abs{\nabla u}}{\sin(\theta)}\right)=\frac{\csc(\theta)}{\abs{\nabla u}}\left(\overline{\nabla}^2u(\cdot,\nabla u)-\cot(\theta)\abs{\nabla u}\left(\nabla u+\abs{\nabla u}\nabla\theta\right)\right).
        \end{equation}
        Taking the norm of both sides, applying the Cauchy-Schwarz inequality, and using the triangle inequality leads to the following inequality:
        \begin{equation}
            \abs{\nabla\frac{\abs{\nabla u}}{\sin(\theta)}}\leq\csc(\theta)\abs{\overline{\nabla}^2u}+\abs{\cot(\theta)}\csc(\theta)\abs{\nabla u+\abs{\nabla u}\nabla \theta}.
        \end{equation}
        We now observe that $\abs{\cot{(\theta)}}\leq\csc(\theta)$, which leads to the following inequality:
        \begin{equation}
            \abs{\nabla\frac{\abs{\nabla u}}{\sin(\theta)}}\leq\csc(\theta)\abs{\overline{\nabla}^2u}+\csc^{2}(\theta)\abs{\nabla u+\abs{\nabla u}\nabla\theta}.
        \end{equation}
        Upon integrating both sides of the above and using Lemma \ref{lem:u-theta_alignment}, we get the result.
    \end{proof}

    Now the bound on the Cheeger constant in Definition \ref{defn:FamilyOfMetrics} combined with the integral gradient bound of Lemma \ref{lem:GradRatio_u-sin_bound} leads to an estimate which again compares the gradient of $u$ to the gradient of the Llarull potential on the round sphere, $\sin(\theta)$.
    \begin{cor}\label{cor:Ratio_grad_u-grad_sin-Poincare}
        There exists a constant $\Cl{const:Ratio_grad_u-grad_sin-Poincare}(\overline{m},\Lambda)$ so that for any $(\Sp^3,g)$ in $\Sc(V,D,\overline{m},\Lambda)$ with Llarull potential $u$ centered at $\pm p$, there exists a a constant $a(g)\geq0$ such that
        \begin{equation}
            \int_{\Sp^3}\abs{\frac{\abs{\nabla u}}{\sin(\theta)}-a(g)}dV_{g}\leq\Cr{const:Ratio_grad_u-grad_sin-Poincare}(\overline{m},\Lambda)\left\|\left(6-R_{g}\right)^{+}\right\|^{\frac12}_{L^{2}(g)}
        \end{equation}
        and
        \begin{equation}
            \abs{\set{x:\abs{\frac{\abs{\nabla u(x)}}{\sin(\theta(x))}-a(g)}>\tau}}_g\leq\frac{1}{\tau}\Cr{const:Ratio_grad_u-grad_sin-Poincare}(\overline{m},\Lambda)\left\|\left(6-R_{g}\right)^{+}\right\|^{\frac12}_{L^{2}(g)}.
        \end{equation}
    \end{cor}
    \begin{proof}
        The second inequality follows from the first by an application of Chebyshev's inequality, see for instance \cite{folland1999}.
        Let us first begin by observing that $\tfrac{\abs{\nabla u}}{\sin(\theta)}$ is in $L^{1}(g)$. To see this, apply Lemma \ref{lem:existence} to find a constant $C$ (depending on $g$) so that $\tfrac{\abs{\nabla u}}{\sin(\theta)} \le C\csc(\theta)$, and then note that $\csc(\theta)$ is integrable on any Riemannian $3$-sphere. In light of Lemma \ref{lem:GradRatio_u-sin_bound}, it follows that $\frac{\abs{\nabla u}}{\sin(\theta)}\in W^{1,1}(g)$. Next, let's recall that $IN_{1}(g)=SN_{1}(g)$, and that by definition
        \begin{equation}
            SN_{1}(g)=\inf\set{\frac{\int_{\Sp^3}\abs{\nabla f}dV_{g}}{\inf_{k\in\R}\int_{\Sp^3}\abs{f-k}dV_{g}}:f\in W^{1,1}(g)}.
        \end{equation}
         Therefore, we have that
        \begin{equation}
            \inf_{k\in\R}\int_{\Sp^3}\abs{\frac{\abs{\nabla u}}{\sin(\theta)}-k}dV_{g}\leq \frac{1}{SN_{1}(g)}\int_{\Sp^3}\abs{\nabla\frac{\abs{\nabla u}}{\sin(\theta)}}dV_{g}.
        \end{equation}
        Let $a(g)$ be any nonnegative constant which achieves the infimum above; the constant can be chosen to be non-negative since $\frac{\abs{\nabla u}}{\sin\theta}$ is non-negative.
        As $SN_{1}(g)=IN_{1}(g)\geq\Lambda$, the result now follows from Lemma \ref{lem:GradRatio_u-sin_bound}.
    \end{proof}
    Again, by the bound on the Cheeger constant in Definition \ref{defn:FamilyOfMetrics} we find an estimate which compares $u$ to the  Llarull potential on the round sphere, $\cos(\theta)$.
    \begin{cor}\label{cor:u_cos_alignment}
        There exists a constant $\Cl{const:u_cos_alignment}(\overline{m},\Lambda)$ so that for any $(\Sp^3,g)$ in $\Sc(V,D,\overline{m},\Lambda)$ with Llarull potential $u$ centered at $\pm p$, there exists constants $a(g)$ and $\sigma(g)$ such that
        \begin{equation}
            \int_{\Sp^{3}}\abs{u-a(g)\cos(\theta)-\sigma(g)}dV_{g}\leq \Cr{const:u_cos_alignment}(\overline{m},\Lambda)\left\|\left(6-R_{g}\right)^{+}\right\|_{L^{2}(g)}^{\frac12}
        \end{equation}
        and
        \begin{equation}
            \abs{\set{x:\abs{u(x)-a(g)\cos(\theta(x))-\sigma(g)}>\tau}}\leq\frac{1}{\tau}\Cr{const:u_cos_alignment}(\overline{m},\Lambda)\left\|\left(6-R_{g}\right)^{+}\right\|_{L^{2}(g)}^{\frac12}.
        \end{equation}
    \end{cor}
    \begin{proof}
        The second inequality will follow from the first after an application of Chebyshev's inequality.
        Let $a(g)$ be as in Corollary \ref{cor:Ratio_grad_u-grad_sin-Poincare}, and set $f=u-a\cos(\theta)$. Once we establish an estimate for the $L^{1}$ norm of $\nabla f$, the proof is the same as in Corollary \ref{cor:Ratio_grad_u-grad_sin-Poincare}.

        To begin, we see that
        \begin{equation}
            \nabla f=\nabla u+a(g)\sin(\theta)\nabla\theta.
        \end{equation}
        Adding and subtracting $\abs{\nabla u}\nabla\theta$, and then taking the norm gives
        \begin{equation}\label{eq:uacor1}
            \abs{\nabla f}\leq\abs{\nabla u+\abs{\nabla u}\nabla\theta}+\abs{a(g)\sin(\theta)-\abs{\nabla u}}\abs{\nabla\theta}.
        \end{equation}
        Using the fact that $\abs{\nabla\theta}\leq 1$, integrating \eqref{eq:uacor1} gives us
        \begin{equation}
            \int_{\Sp^3}\abs{\nabla f}dV_g\leq\int_{\Sp^3}\abs{\nabla u+\abs{\nabla u}\nabla\theta}dV_{g}+\int_{\Sp^3}\sin(\theta)\abs{a(g)-\frac{\abs{\nabla u}}{\sin(\theta)}}dV_{g}.
        \end{equation}
        Now we may apply Lemma \ref{lem:u-theta_alignment} and Corollary \ref{cor:Ratio_grad_u-grad_sin-Poincare} with the trivial bounds $\sin(\theta)\leq 1$ and $\csc^2(\theta)\geq1$ to see that
        \begin{align}
        \begin{split}
            \int_{\Sp^3}\abs{\nabla f}dV_{g}&\leq 
            \int_{\Sp^3}\csc^2(\theta)\abs{\nabla u+\abs{\nabla u}\nabla\theta}dV_{g}+\int_{\Sp^3}\abs{a(g)-\frac{\abs{\nabla u}}{\sin(\theta)}}dV_{g}\\
            &\leq \Cr{const:u-theta_alignment}(\overline{m})\left\|\left(6-R_{g}\right)^{+}\right\|^{\frac12}_{L^{2}(g)}+\Cr{const:Ratio_grad_u-grad_sin-Poincare}(\overline{m},\Lambda)\left\|\left(6-R_{g}\right)^{+}\right\|^{\frac12}_{L^{2}(g)}.
        \end{split}
        \end{align}
    \end{proof}
    \section{Polar Estimates} \label{sect: polar estimates}
    Now that we have established several $L^p$ and $W^{1,1}$ bounds for Llarull potentials, and have shown that they don't deviate too much from functions of the form $a\cos(\theta)$, we would like to know more about the constant $a$. In order to do this, we proceed with a finer analysis of the behavior of Llarull potentials near the poles.
    \begin{lem}\label{lem:csc2_norm_grad_u_bndry}
        There exists a constant $\Cl{const:csc2_norm_grad_u_bndry}(\overline{m})$ so that for any $(\Sp^3,g)$ in $\Sc(V,D,\overline{m},\Lambda)$ with Llaurull potential $u$ centered at $\pm p$, there exists constants $\sigma_{p}$ and $\sigma_{-p}$ in $[\frac{\pi}{8},\frac{\pi}{4}]$, depending on $g$, such that
        \begin{align}
            &\int_{\partial B^{\Sp}(p,\sigma_p)}\abs{\nabla u}d A_{g}\leq\Cr{const:csc2_norm_grad_u_bndry}(\overline{m}),
            \\
            &\int_{\partial B^{\Sp}(-p,\sigma_{-p})}\abs{\nabla u}d A_{g}\leq\Cr{const:csc2_norm_grad_u_bndry}(\overline{m}).
        \end{align}
    \end{lem}
    \begin{proof}
        From inequality \eqref{eq:csc2_norm_grad_u} and Lemma \ref{lem:L2Grad_u_bound} we have
        \begin{equation}\label{MassConsequence}
            \int_{B^{\Sp}(p,\tfrac{\pi}{4})\setminus B^{\Sp}(p,\tfrac{\pi}{8})}\csc^2(\theta)\abs{\nabla u}dV_{g}\leq\int_{\Sp^3}\csc^2(\theta)\abs{\nabla u}dV_{g}\leq 8\pi+\frac{1}{2}m(g)\Cr{const:L2Grad_u_bound}(\overline{m}).
        \end{equation}
        Applying the co-area formula to the left most integral above gives us the following equality: 
        \begin{align}
        \begin{split}
            \int_{\tfrac{\pi}{8}}^{\tfrac{\pi}{4}}\csc^2(t)\left(\int_{\partial B^{\Sp}(p,t)}\abs{\nabla u}d A_{g}\right)dt&=\int_{B^{\Sp}(p,\tfrac{\pi}{4})\setminus B^{\Sp}(p,\tfrac{\pi}{8})}\csc^2(\theta)\abs{\nabla u}\abs{\nabla\theta}dV_{g}
            \\
            &\leq8\pi+\frac{1}{2}\overline{m}\Cr{const:L2Grad_u_bound}(\overline{m}),
        \end{split}
        \end{align}
        where the last inequality comes from the fact that $\abs{\nabla\theta}\leq 1$ and \eqref{MassConsequence}. By the Mean Value Theorem, it follows that there exists $\sigma_{p}$ in $[\tfrac{\pi}{8},\tfrac{\pi}{4}]$ such that
        \begin{equation}
            \csc^2(\sigma_{p})\int_{\partial B^{\Sp}(p,\sigma_p)}\abs{\nabla u}dA_{g}\leq\frac{8}{\pi}\left(8\pi+\frac{1}{2}\overline{m}\Cr{const:L2Grad_u_bound}(\overline{m})\right).
        \end{equation}
        An analogous argument centered around the pole $-p$ establishes the second inequality.
    \end{proof}

    Since $\csc(\theta)$ blows up at the poles, inequality \eqref{eq:csc2_norm_grad_u} suggests that the gradients of Llarull potentials should be relatively small near the poles. We will now determine in which sense the above intuition is true. The first step is to show that we can improve the bound in inequality \eqref{eq:csc2_norm_grad_u} as follows.
    \begin{lem}\label{lem:csc3_norm_grad_u}
        There is a constant $\Cl{const:csc3_norm_grad_u}(\overline{m})$ so that for any $(\Sp^3,g)$ in $\Sc(V,D,\overline{m},\Lambda)$ with Llarull potential $u$ centered at $\pm p$ and any radius $r\leq\tfrac{\pi}{8}$, we have
        \begin{align}
            &\int_{B^{\Sp}(p,r)}\csc^{3}(\theta)\abs{\nabla u}dV_{g}\leq\Cr{const:csc3_norm_grad_u}(\overline{m}),\label{eq-lemcsc3:1}
            \\
            &\int_{B^{\Sp}(-p,r)}\csc^{3}(\theta)\abs{\nabla u}dV_{g}\leq\Cr{const:csc3_norm_grad_u}(\overline{m})\label{eq-lemcsc3:2}.
        \end{align}
    \end{lem}
    \begin{proof}
        Since the proofs of \eqref{eq-lemcsc3:1} and \eqref{eq-lemcsc3:2} are nearly identical, we will focus our attention near $p$ and prove only \eqref{eq-lemcsc3:1}. 
        Let $\sigma_{p}$ be as in Lemma \ref{lem:csc2_norm_grad_u_bndry}, set $\Omega_{i}=B^{\Sp}(p,\sigma_p)\setminus B^{\Sp}\left(p,\tfrac{1}{i}\right)$, and denote the outward unit normal to $\partial \Omega_i$ by $n$.
        According to Lemma \ref{lem:existence}, the Llarull potential has a Lipschitz bound which only depends on the metric $g$. Therefore, we have
        \begin{align}\label{eq:Grad_ui_bndry}
            \abs{\nabla u}\Big\vert_{\partial B^\Sp\left(p,\tfrac{1}{i}\right)}\leq C,\quad          \abs{\nabla u}\Big\vert_{\partial B^\Sp\left(-p,\tfrac{1}{i}\right)}\leq C
        \end{align}
        for some constant $C$.
        
        Fix $\varepsilon$ to be any number in the range $[0,1)$. We may multiply both sides of the Llarull potential equation for $u_i$ by $\csc^{1+\varepsilon}(\theta)$ to obtain
        \begin{equation}
            \csc^{1+\varepsilon}(\theta)\Delta u_i=-3\cos(\theta)\csc^{2+\varepsilon}(\theta)\abs{\nabla u_i}.
        \end{equation}
        Integrating the above equality over the domain $\Omega_i$, integrating by parts, and then rearranging terms gives
        \begin{align}
        \begin{split}
            \int_{\partial B^{\Sp}(p,\sigma_p)\cup\partial B^\Sp\left(p,\tfrac{1}{i}\right)}\csc^{1+\varepsilon}(\theta)\frac{\partial u}{\partial n}dA_{g}=&\int_{\Omega_i}g\left(\nabla\csc^{1+\varepsilon}(\theta),\nabla u\right)dV_{g}
            \\
            &-3\int_{\Omega_i}\cos(\theta)\csc^{2+\varepsilon}(\theta)\abs{\nabla u}dV_{g}.
        \end{split}
        \end{align}
        Calculating $\nabla\csc^{1+\varepsilon}(\theta)$ and further rearranging terms gives us the following:
        \begin{align}
        \begin{split}
            \csc^{1+\varepsilon}(\sigma_{p})\int_{\partial B^{\Sp}(p,\sigma_p)}\frac{\partial u}{\partial n}dA_{g}=&-\int_{\partial B^{\Sp}\left(p,\tfrac{1}{i}\right)}\csc^{1+\varepsilon}(\theta)\frac{\partial u}{\partial n}dA_g
            \\
            &-(1+\varepsilon)\int_{\Omega_i}\cos(\theta)\csc^{2+\varepsilon}(\theta)g\left(\nabla\theta,\nabla u\right)dV_{g}
            \\
            &-3\int_{\Omega_i}\cos(\theta)\csc^{2+\varepsilon}(\theta)\abs{\nabla u}dV_{g}.
        \end{split}
        \end{align}
        Taking the absolute value of both sides, using Lemma \ref{lem:csc2_norm_grad_u_bndry}, using the fact that $\sigma_p\in[\tfrac{\pi}{8},\tfrac{\pi}{4}]$, and applying the triangle inequality yields 
        \begin{align}
        \begin{split}
            &\Cl{const:intermediary}(\overline{m})\\
            &\geq\abs{3\int_{\Omega_{i}}\cos(\theta)\csc^{2+\varepsilon}(\theta)\abs{\nabla u}dV_{g}}
            \\
            &-\abs{(1+\varepsilon)\int_{\Omega_{i}}\cos(\theta)\csc^{2+\varepsilon}(\theta)g\left(\nabla\theta,\nabla u\right)dV_{g}+\int_{\partial B^\Sp(p,\tfrac{1}{i})}\csc^{1+\varepsilon}(\theta)\frac{\partial u}{\partial n}dA_g}
        \end{split}
        \end{align}
        where $\Cr{const:intermediary}(\overline{m})$ is a constant depending only on $\overline{m}$.
        Then, we may estimate the second term on the right by the triangle inequality to find
        \begin{align}\label{eq:csc_grad_u_i_bound_1}
        \begin{split}
            \Cr{const:intermediary}(\overline{m})&\geq\abs{3\int_{\Omega_{i}}\cos(\theta)\csc^{2+\varepsilon}(\theta)\abs{\nabla u}dV_{g}}
            \\
            &-\abs{(1+\varepsilon)\int_{\Omega_{i}}\cos(\theta)\csc^{2+\varepsilon}(\theta)g\left(\nabla\theta,\nabla u\right)dV_{g}}
            \\
            &-\abs{\int_{\partial B^{\Sp}\left(p,\tfrac{1}{i}\right)}\csc^{1+\varepsilon}(\theta)\frac{\partial u}{\partial n}dA_g}.
        \end{split}
        \end{align}
        
        Before estimating further, we observe the following two inequalities. First, it follows from the Cauchy-Schwarz inequality, the fact that $\abs{\nabla\theta}_g\leq1$, and the fact that $\cos(\theta)$ is non-negative on $\Omega_i$, that we have
        \begin{align}\label{eq:intermediateest1}
        \begin{split}
            &\abs{(1+\varepsilon)\int_{\Omega_{i}}\cos(\theta)\csc^{2+\varepsilon}(\theta)g\left(\nabla\theta,\nabla u\right)dV_{g}}\\
            &\leq(1+\varepsilon)\int_{\Omega_{i}}\cos(\theta)\csc^{2+\varepsilon}(\theta)\abs{\nabla u}dV_{g}.
        \end{split}
        \end{align}
        Second, the gradient bound \eqref{eq:Grad_ui_bndry} implies
        \begin{equation}\label{eq:IntermediateEstimateEq}
            \abs{\int_{\partial B^{\Sp}\left(p,\tfrac{1}{i}\right)}\csc^{1+\varepsilon}(\theta)\frac{\partial u}{\partial n}dA_{g}}\leq C\int_{\partial B^\Sp\left(p,\tfrac{1}{i}\right)}\csc^{1+\varepsilon}(\theta)dA_{g}.
        \end{equation}
        There exists a constant $C'>0$ depending on $g$ such that $g_{\Sp^3}\leq g\leq C' g_{\Sp^3}$, which allows us to write $\csc(\theta(x))\leq\csc(d_{g}(p,x)/C')$, and $|\partial B^{\Sp}(p,\tfrac{1}{i})|_g\leq C'\sin^2(\tfrac{1}{i})$. Apply{\-}ing these statements to \eqref{eq:IntermediateEstimateEq} gives us
        \begin{equation}\label{eq:intermediateest2}
            \abs{\int_{\partial B^{\Sp}(p,\tfrac{1}{i})}\csc^{1+\varepsilon}(\theta)\frac{\partial u}{\partial n}dA_{g}}\leq 4\pi C C'\csc^{1+\varepsilon}\left(\frac{1}{iC'}\right)\sin^{2}\left(\frac{1}{i}\right).
        \end{equation}
        Combining \eqref{eq:intermediateest1} and \eqref{eq:intermediateest2} with \eqref{eq:csc_grad_u_i_bound_1} yields
        \begin{align}\label{eq:csc_grad_u_i_bound_2}
        \begin{split}
            \Cr{const:intermediary}(\overline{m})\geq&(2-\varepsilon)\int_{\Omega_{i}}\cos(\theta)\csc^{2+\varepsilon}(\theta)\abs{\nabla u}dV_{g}
            \\
            &-4\pi CC'\csc^{1+\varepsilon}\left(\frac{1}{iC'}\right)\sin^{2}\left(\frac{1}{i}\right).
        \end{split}
        \end{align}
    Applying the $\limsup$ to both sides of \eqref{eq:csc_grad_u_i_bound_2} as $i\to\infty$ and leveraging Lemma \ref{lem:csc2_norm_grad_u_bndry} gives
        \begin{align}\label{eq:postintermediate1}
        \begin{split}
            \Cr{const:intermediary}(\overline{m})\geq&(2-\varepsilon)\limsup_{i\rightarrow\infty}\int_{\Omega_{i}}\cos(\theta)\csc^{2+\varepsilon}(\theta)\abs{\nabla u_i}dV_{g}
            \\
            &-4\pi CC'\limsup_{i\rightarrow\infty}\csc^{1+\varepsilon}\left(\frac{1}{iC'}\right)\sin^{2}\left(\frac{1}{i}\right).
        \end{split}
        \end{align}
        Since $\varepsilon<1$, the last term on the right side of \eqref{eq:postintermediate1} is actually zero and, crucially, the constants depending on $g$ completely disappear. Rearranging terms, \eqref{eq:postintermediate1} implies
        \begin{equation}
            \Cr{const:intermediary}(\overline{m})\geq\limsup_{i\rightarrow\infty}\int_{\Omega_{i}}\cos(\theta)\csc^{2+\varepsilon}(\theta)\abs{\nabla u}dV_{g}.
        \end{equation}
        Finally, the location of the set $\Omega_i$ implies $\cos(\theta)\big\vert_{\Omega_{i}}\geq\cos\left(\frac{\pi}{4}\right)$, and so
        \begin{equation}\label{eq:csc_grad_u_i_bound_3}
            \sqrt{2}\Cr{const:intermediary}(\overline{m})\geq\limsup_{i\rightarrow\infty}\int_{\Omega_{i}}\csc^{2+\varepsilon}(\theta)\abs{\nabla u}dV_{g}\geq\int_{B^{\Sp}(p,\sigma_{p})}\csc^{2+\varepsilon}(\theta)\abs{\nabla u}dV_{g}.
        \end{equation}

    Noting that $\csc^{2+\varepsilon}(\theta)\abs{\nabla u}$ monotonically increases to $\csc^{3}(\theta)\abs{\nabla u}$ as $\varepsilon$ approaches $1$, the Monotone Convergence Theorem may be applied to find
        \begin{align}
        \begin{split}
            \sqrt{2}\Cr{const:intermediary}(\overline{m})&\geq\lim_{\varepsilon\rightarrow 1}\int_{B^{\Sp}(p,\sigma_{p})}\csc^{2+\varepsilon}(\theta)\abs{\nabla u}dV_{g}
            \\
            &=\int_{B^{\Sp}(p,\sigma_{p})}\csc^{3}(\theta)\abs{\nabla u}dV_{g}.
        \end{split}
        \end{align}
        The result now follows since $\sigma_{p}\geq\tfrac{\pi}{8}$.
    \end{proof}

    The lemma above suggests that around the poles, the gradient of the Llarull potential $u$ should be nearly bounded, and in fact close to zero. Since $u(\pm p)=\pm1$, this means that $u$ should be roughly equal to $\pm1$ around the poles. Here we show that if we take averages with respect to the round metric, then the above reasoning holds.
    \begin{lem}\label{lem:avg_polar_control_of_u}
        There exists a constant $\Cl{const:avg_polar_control_of_u}(\overline{m})$ so that for any $(\Sp^3,g)$ in $\Sc(V,D,\overline{m},\Lambda)$ with Llarull potential $u$ centered at $\pm p$ and $t\in(0,\tfrac{\pi}{8}]$, we have
        \begin{align}
            1-\frac{1}{4\pi}\int_{\partial B^\Sp(p,t)}udA_{\Sp}&\leq\Cr{const:avg_polar_control_of_u}(\overline{m})\sin(t)\label{eq:avg_polar_control}
            \\
            -1-\frac{1}{4\pi}\int_{\partial B^\Sp(-p,t)}udA_{\Sp}&\geq-\Cr{const:avg_polar_control_of_u}(\overline{m})\sin(t)
        \end{align}
        where $dA_\Sp$ denotes the area form of the unit $2$-sphere.
    \end{lem}
    \begin{proof}
        Once again, since the calculations will be the same at both poles, we will focus on the pole $p$. Let $F_{t}$ be the flow generated by the vector field $\partial_{\theta}=\frac{1}{\abs{\nabla\theta}^2}\nabla\theta$. For any $t>0$ we have the following equality
        \begin{align}
        \begin{split}
             \frac{d}{dt}\int_{\partial B^{\Sp}(p,t)}udA_{\Sp}&=\frac{d}{ds}\Big\vert_{s=0}\int_{F_{s}\left(\partial B^{\Sp}(p,t)\right)}udA_{\Sp}\\
             &=\int_{\partial B^{\Sp}(p,t)}du\left(\Dot{F}_{t}\right)dA_{\Sp}=\int_{\partial B^{\Sp}(p,t)}g\left(\nabla u,\frac{\nabla\theta}{\abs{\nabla\theta}^2}\right)dA_{\Sp},\label{eq:Flow2}
             \end{split}
        \end{align}
        where $\Dot{F}_t$ denotes the partial derivative in the $t$ variable. We emphasize that the area form present in the above expressions is not the one induced on $\partial B^{\Sp}(p,t)$, but the area form of the round unit $2$-sphere.
       
        Taking absolute values of \eqref{eq:Flow2}, and then applying the Cauchy-Schwarz inequality results in
        \begin{equation}
            \abs{\frac{d}{dt}\int_{\partial B^{\Sp}(p,t)}udA_{\Sp}}\leq\int_{\partial B^{\Sp}(p,t)}\frac{\abs{\nabla u}}{\abs{\nabla\theta}}dA_{\Sp}.
        \end{equation}
        Multiplying and dividing the right hand side by $\sin^{3}(t)$ shows that
        \begin{align}
        \begin{split}
            \abs{\frac{d}{dt}\int_{\partial B^{\Sp}(p,t)}udA_{\Sp}}&\leq\sin(t)\int_{\partial B^{\Sp}(p,t)}\csc^{3}(t)\frac{\abs{\nabla u}}{\abs{\nabla\theta}}\sin^{2}(t)dA_{\Sp}\\
            {}&\leq \sin(t)\int_{\partial B^{\Sp}(p,t)}\csc^{3}(t)\frac{\abs{\nabla u}}{\abs{\nabla\theta}}dA_{g},\label{eq:Flowarea}
        \end{split}
        \end{align}
        where we have used $dA_{g}\vert_{\partial B^{\Sp}(p,t)}\geq\sin^{2}(t)dA_{\Sp}$ which itself follows from the fact $g\geq g_{\Sp^3}$.
        Integrating both sides of \eqref{eq:Flowarea} in $t$ and bounding $\sin(t)$ by $\sin(s)$, where $s\leq\tfrac{\pi}{4}$, leads us to
        \begin{align}
        \begin{split}
            \int_{t_0}^{s}\abs{\frac{d}{dt}\int_{\partial B^{\Sp}(p,t)}udA_{\Sp}}dt&\leq\sin(s)\int_{t_0}^{s}\int_{\partial B^{\Sp}(p,t)}\csc^{3}(t)\frac{\abs{\nabla u}}{\abs{\nabla\theta}}dA_{g}dt
            \\
            &=\sin(s)\int_{\theta^{-1}[t_0,s]}\csc^{3}(\theta)\abs{\nabla u}dV_{g},
        \end{split}
        \end{align}
        where the last equality comes from the co-area formula. Therefore, from the Fundamental Theorem of Calculus and Lemma \ref{lem:csc3_norm_grad_u}, we get
        \begin{equation}\label{eq:FlowFTC}
            \abs{\int_{\partial B^{\Sp}(p,s)}udA_{\Sp}-\int_{\partial B^{\Sp}(p,t_0)}udA_{\Sp}}\leq\Cr{const:csc3_norm_grad_u}(\overline{m})\sin(s).
        \end{equation}
        Finally, since $u(p)=1$ and $u$ is continuous, it follows that
        \begin{equation}
            \lim_{t_0\rightarrow0}\int_{\partial B^{\Sp}(p,t_0)}udA_{\Sp}=4\pi.
        \end{equation}
        Combining this with \eqref{eq:FlowFTC} and the fact that $u\leq1$, and then rearranging terms, gives the desired inequality \eqref{eq:avg_polar_control}.
    \end{proof}
    We now proceed to apply Lemma \ref{lem:avg_polar_control_of_u} to estimate the volume with respect to $g_{\Sp^3}$ of $B^{\Sp}(\pm p,r)$ intersected with sub- and super-level sets of $u$, respectively.
\begin{cor}\label{cor:u_polar_lower_bound}
    There exists a constant $\Cl{const:u_polar_lower_bound}(\overline{m})$ so that for any $(\Sp^3,g)$ in $\Sc(V,D,\overline{m},\Lambda)$ with Llarull potential $u$ centered at $\pm p$, the following hold for $r\in[0,\frac{\pi}{8}]$ and $\gamma\in[0,1)$:
        \begin{align}
            &\abs{B^{\Sp}(p,r)\cap\set{u\leq\gamma}}_{g_{\Sp^3}}\leq\frac{4\pi\Cr{const:u_polar_lower_bound}(\overline{m})}{1-\gamma}\int_{0}^{r}\sin^3(t)dt
            \\
            &\abs{B^\Sp(-p,r)\cap\set{u\geq-\gamma}}_{g_{\Sp^3}}\leq\frac{4\pi\Cr{const:u_polar_lower_bound}(\overline{m})}{1-\gamma}\int_{0}^{r}\sin^3(t)dt.
    \end{align}
\end{cor}
\begin{proof}
    The proof is nearly identical for the pole $-p$, and so we focus only on the pole $p$. As usual, $dA_\Sp$ denotes the area form of the unit $2$-sphere. Rearranging the result of Lemma \ref{lem:avg_polar_control_of_u} gives
    \begin{equation}\label{eq:FlowCor0}
            \frac{1}{4\pi}\int_{\partial B^{\Sp}(p,t)}udA_\Sp\geq1-\Cr{const:avg_polar_control_of_u}(\overline{m})\sin(t).
    \end{equation}
    We also have
    \begin{align}
    \begin{split}
        \frac{1}{4\pi}\int_{\partial B^{\Sp}(p,t)}udA_\Sp&=\frac{1}{4\pi}\int_{\partial B^{\Sp}(p,t)\cap\set{u\leq\gamma}}u dA_\Sp+\frac{1}{4\pi}\int_{\partial B^{\Sp}(p,t)\cap\set{u>\gamma}}udA_\Sp\\
        &\leq\frac{\gamma}{4\pi}\abs{\partial B^{\Sp}(p,t)\cap\set{u\leq\gamma}}_{g_{\Sp^2}}+\frac{1}{4\pi}\int_{\partial B^{\Sp}(p,t)\cap\set{u>\gamma}}udA_\Sp.\label{eq:FlowCor1}
        \end{split}
    \end{align}
    Using the fact that $u\leq1$ to estimate the last term of \eqref{eq:FlowCor1}, inequality \eqref{eq:FlowCor0} implies
    \begin{equation}
            \frac{\gamma}{4\pi}\abs{\partial B^{\Sp}(p,t)\cap\set{u\leq\gamma}}_{g_{\Sp^2}}+\frac{1}{4\pi}\abs{\partial B^{\Sp}(p,t)\cap\set{u>\gamma}}_{g_{\Sp^2}}\geq 1-\Cr{const:avg_polar_control_of_u}(\overline{m})\sin(t).
    \end{equation}
    Since $\set{u>\gamma}$ is the complement of $\set{u\leq\gamma}$, the last inequality can be rewritten as
    \begin{equation}
            1+\frac{\gamma-1}{4\pi}\abs{\partial B^{\Sp}(p,t)\cap \set{u\leq\gamma}}_{g_{\Sp^2}}\geq1-\Cr{const:avg_polar_control_of_u}(\overline{m})\sin(t).
    \end{equation}
    A final rearrangement of the above line, and then an integration, gives the result.
    \end{proof}
    
    
\section{Volume Above Distance Below Convergence} \label{sect: VADB convergence}

Our goal in this section is to finish the proof of Theorem \ref{thm-Main} by applying Theorem \ref{thm-VADB}. This argument focuses on studying sets wherein the Llarull potential nearly satisfies certain analytical features of the model situation on the round sphere.
    
First, some notation: fix $(\Sp^3,g)$ in $\Sc(V,D,\overline{m},\Lambda)$ with a Llarull potential $u$ centered at give poles $\pm p$. Let $a(g)$ and $\sigma(g)$ be the constants appearing in Corollary \ref{cor:u_cos_alignment}. Given $\tau>0$, we consider the set 
    \begin{equation}\label{eq:thegoodset1}
         \widetilde{E}_{\tau,g}:=\set{x\in\Sp^3:\abs{u(x)-a(g)\cos{\theta(x)}-\sigma(g)}\leq\tau}.
     \end{equation} 
We note that $\sigma(g)$ in \eqref{eq:thegoodset1} is an integration constant and plays only a background role. On the other hand, the constant $a(g)$ determines the precise comparison to the model situation on the round sphere.     
    
    \begin{lem}\label{lem:u_amplitude_lower_bound}
        There exists constants $\Cl{const:effective_a_lower_bound_1}(\overline{m})$ and $\Cl{const:effective_a_lower_bound_2}(\overline{m},\Lambda)$ so that for any $(\Sp^3,g)$ in $\Sc(V,D,\overline{m},\Lambda)$ with Llarull potential $u$ centered at $\pm p$, the constant $a(g)$ given by Corollary \ref{cor:u_cos_alignment} satisfies the bound
        \begin{equation}
        a(g)\geq 1-\Cr{const:effective_a_lower_bound_1}(\overline{m})\left\|\left(6-R_{g}\right)^{+}\right\|^{\frac{1}{12}}_{L^{2}(g)}-\Cr{const:effective_a_lower_bound_2}(\overline{m},\Lambda)\left\|\left(6-R_{g}\right)^{+}\right\|^{\frac{1}{4}}_{L^{2}(g)}.
    \end{equation}

    \end{lem}
    \begin{proof}
        For any $r\in[0,\frac{\pi}{2}]$, it is a simple fact that $\sin(r)\geq\frac{2}{\pi}r$. In particular, $\abs{B^{\Sp}(p,r)}_{g_{\Sp^3}}\geq\tfrac{16}{3\pi}r^3$. For $\tau>0$ and $\gamma <1$ to be determined, we study the region $\widetilde{E}_{\tau,g}\cap\left\{u>\gamma\right\}\cap B^{\Sp}(p,r)$. Using the sub-additivity of measures, the volume of this set with respect to the round metric satisfies
        \begin{equation}
        \begin{split}
            \abs{\widetilde{E}_{\tau,g}\cap\left\{u>\gamma\right\}\cap B^{\Sp}(p,r)}_{g_{\Sp^3}}&\geq\abs{B^{\Sp}(p,r)}_{g_{\Sp^3}}
            \\&\quad-\abs{\left\{u\leq\gamma\right\}\cap B^{\Sp}(p,r)}_{g_{\Sp^3}}
            -\abs{\widetilde{E}_{\tau,g}^c}_{g_{\Sp^3}}.
            \end{split}
        \end{equation}
        It follows from Corollary \ref{cor:u_polar_lower_bound} and Corollary \ref{cor:u_cos_alignment}, with the fact that $g\geq g_{\Sp^3}$, that
        \begin{align}\label{eq:abound1}
            \begin{split}
                \abs{\widetilde{E}_{\tau,g}\cap\left\{u>\gamma\right\}\cap B^{\Sp}(p,r)}_{g_{\Sp^3}}\geq&\tfrac{16}{3\pi}r^3-\frac{4\pi\Cr{const:u_polar_lower_bound}(\overline{m})}{1-\gamma}\int_{0}^{r}\sin^3(t)dt
                \\
                &-\frac{\Cr{const:u_cos_alignment}(\overline{m},\Lambda)}{\tau}\left\|(6-R_{g})^{+}\right\|^{\frac12}_{L^{2}(g)}.\\
                \geq&\tfrac{16}{3\pi}r^3-\frac{\pi\Cr{const:u_polar_lower_bound}(\overline{m})}{1-\gamma}r^{4}\\
                &-\frac{\Cr{const:u_cos_alignment}(\overline{m},\Lambda)}{\tau}\left\|(6-R_{g})^{+}\right\|^{\frac12}_{L^{2}(g)}
            \end{split}
        \end{align}
        where we have used $\sin(t)\leq t$ in the last equality.
        A similar argument working near $-p$ shows
        \begin{equation}\label{eq:abound2}
        \begin{split}
            \abs{\widetilde{E}_{\tau,g}\cap\left\{u<-\gamma\right\}\cap B^{\Sp}(-p,r)}_{g_{\Sp^3}}\geq&\tfrac{16}{3\pi}r^3-\frac{\pi\Cr{const:u_polar_lower_bound}(\overline{m})}{1-\gamma}r^{4}
            \\&-\frac{\Cr{const:u_cos_alignment}(\overline{m},\Lambda)}{\tau}\left\|(6-R_{g})^{+}\right\|^{\frac12}_{L^{2}(g)}.           
        \end{split}
        \end{equation}
        
        If $\frac{\pi\Cr{const:u_polar_lower_bound}(\overline{m})}{1-\gamma}r^{4}<\tfrac{16}{6\pi}r^3$ and $\frac{\Cr{const:u_cos_alignment}(\overline{m},\Lambda)}{\tau}\left\|(6-R_{g})^{+}\right\|^{\frac12}_{L^{2}(g)}<\tfrac{16}{6\pi}r^3$, then the sets appearing in the left side of \eqref{eq:abound1} and \eqref{eq:abound2} must have positive measure. We now specify parameters $r$, $\tau$, $\gamma$ so that $\widetilde{E}_{\tau,g}\cap\left\{u>\gamma\right\}\cap B^{\Sp}(p,r)$ and $\widetilde{E}_{\tau,g}\cap\left\{u<-\gamma\right\}\cap B^{\Sp}(-p,r)$ are not empty: choose 
        \begin{align}
        \begin{split}
            &r=\left\|(6-R_{g})^{+}\right\|^{\frac{1}{12}}_{L^{2}(g)}, \quad \tau=\frac{12\pi\Cr{const:u_cos_alignment}(\overline{m},\Lambda)\|(6-R_{g})^{+}\|^{\frac14}_{L^{2}(g)}}{16}, \\
            &\gamma=1-\frac{12\pi^2\Cr{const:u_polar_lower_bound}(\overline{m})r}{16}.
        \end{split}
        \end{align}
        Pick $x^{+}$ and $x^{-}$ in the regions $\widetilde{E}_{\tau,g}\cap\left\{u>\gamma\right\}\cap B^{\Sp}(p,r)$ and $\widetilde{E}_{\tau,g}\cap\left\{u<-\gamma\right\}\cap B^{\Sp}(-p,r)$, respectively. Then, by definition of these sets, we have that
        \begin{equation}\label{eq:alower1}
            \begin{split}
                2\tau&\geq u(x^{+})-a(g)\cos(\theta(x^{+}))-\sigma(g)\\
            {}&\quad\quad-\left(u(x^{-})-a(g)\cos(\theta(x^{-}))-\sigma(g)\right)
            \\
            &\geq 2\gamma+a(g)\left(\cos(\theta(x^{-}))-\cos(\theta(x^{+}))\right)
            \\
            &\geq 2\gamma-2a(g),
            \end{split}
        \end{equation}
        where we have used the fact that $a(g)\geq 0$, $\cos(\theta(x^{+}))\leq 1$, and $\cos{\theta(x^{-})}\geq-1$.
        Given our choices of $\tau$ and $\gamma$, rearranging \eqref{eq:alower1} finishes the proof.
    \end{proof}

Having established control on $a(g)$, we now introduce the set of primary interest:
\begin{equation}\label{eq:thegoodset2}
  E_{\tau,g}= \set{x\in\Sp^3:\abs{\frac{\abs{\nabla u(x)}}{\sin(\theta(x))}-a(g)}\le\tau}.
\end{equation}
Informally, for small $\tau$ the sets $E_{\tau,g}$ are well-controlled and allow us to effectively compare volumes measured using $g$ to those measured with $g_{\Sp^3}$. Instead of analyzing these regions directly, for $t>0$ we consider
\begin{equation}\label{eq:thegoodset3}
    E_{\tau,g,t}=E_{\tau,g}\cap\left(\Sp^3\setminus\left(B^{\Sp}(p,t)\cup B^{\Sp}(-p,t)\right)\right)
\end{equation} 
which is the portion of $E_{\tau,g}$ lying away from the poles, see Figure \ref{fig-Etau}.

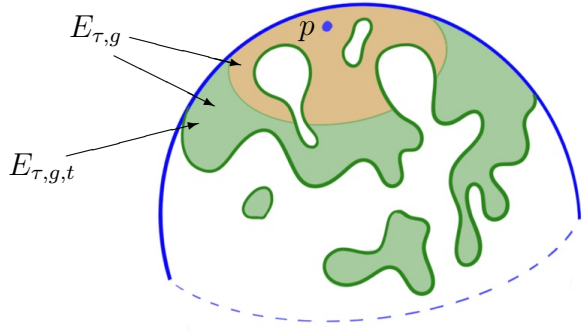
\begin{figure}[h] 
   \center{\includegraphics[width=.5\textwidth]{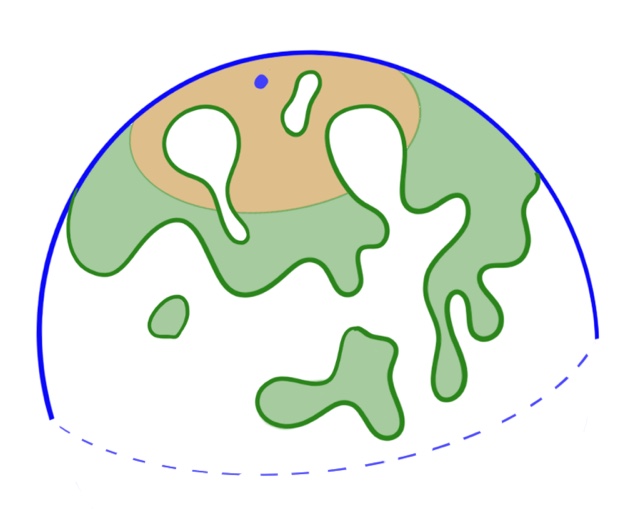}}
   \begin{picture}(0,0)
\put(-120,118){\Large{$p$}}
\put(-207,117){{\Large{$E_{\tau,g}$}}}
\put(-190,110){\vector(2,-1){38}}
\put(-184,117){\vector(4,-1){42}}
\put(-230,65){\Large{$E_{\tau,g,t}$}}
\put(-208,72){\vector(4,1){50}}
\end{picture}
\caption{The set $E_{\tau,g}$ is pictured in orange and green. The portion appearing in green represents $E_{\tau,g,t}$.}
   \label{fig-Etau}
\end{figure}

\begin{lem}\label{lem:vol_Etaut0}
    Assume $\overline{m}\leq1$. There is a constant $\Cl{const:vol_Etaut0}(\overline{m},V,\Lambda)$ so that for any $(\Sp^3,g)$ in $\Sc(V,D,\overline{m},\Lambda)$ with a Llarull potential $u$ centered at $\pm p$, the following holds
    \begin{equation}\label{eq:volEtaut0}
        0\leq|E_{\tau,g,t}|_{g}-|E_{\tau,g,t}|_{g_{\Sp^3}}\leq \Cr{const:vol_Etaut0}(\overline{m},V,\Lambda)\left\|\left(6-R_{g}\right)^{+}\right\|_{L^{2}(g)}^{1/48}
    \end{equation}
    for $\tau=\left\|\left(6-R_{g}\right)^{+}\right\|_{L^{2}(g)}^{1/4}$ and $t=\left\|\left(6-R_{g}\right)^{+}\right\|_{L^{2}(g)}^{1/48}$.

\end{lem}
\begin{proof}
    We begin with some preliminary observations. For convenience, we adopt the notation $m=\left\|\left(6-R_{g}\right)^{+}\right\|^{1/2}_{L^{2}(g)}$. It follows from Corollary \ref{cor:Ratio_grad_u-grad_sin-Poincare} that the compliment of $E_{\tau,g}$ satisfies
    \begin{equation}
        |E_{\tau,g}^c|_{g}\leq\frac{\Cr{const:Ratio_grad_u-grad_sin-Poincare}(\overline{m},\Lambda)}{\tau}\left\|\left(6-R_{g}\right)^{+}\right\|_{L^{2}(g)}^{1/2}=\Cr{const:Ratio_grad_u-grad_sin-Poincare}(\overline{m},\Lambda)\sqrt{m}.
    \end{equation}
    Our choices of $t$ and $\tau$ ensure that there is a constant $\Cr{const:smallanglecsc}(\overline{m})$ so that
    \begin{align}
            &|E_{\tau,g,t}^c|_{g}\leq\Cr{const:Ratio_grad_u-grad_sin-Poincare}(\overline{m},\Lambda)\sqrt{m}+\frac{8\pi}{3}m^{\frac{1}{8}}\label{eq:volEtaut1},
            \\
            &\tau\csc(t)\abs{\Sp^3}_{g}\leq \Cl{const:smallanglecsc}(\overline{m})m^{\frac{11}{24}}V\label{eq:tau_csct_bound},
    \end{align}
   where we have used the fact that $\overline{m}\leq 1$ to bound $\csc(t)$, the $g$-volume of $\Sp^3$ is bounded above by $V$, and $\abs{B^{\Sp}(\pm p,r)}_{g_{\Sp^3}}\leq\frac{4\pi}{3}r^3$ for any $r$. 
   
   Next, we prove a string of inequalities starting from Lemma \ref{lem:existence} which will eventually lead us to the desired estimate \eqref{eq:volEtaut0}. Combining inequality \eqref{eq:csc2_norm_grad_u} in Lemma \ref{lem:existence} with H{\"o}lder's inequality and Lemma \ref{lem:L2Grad_u_bound}, we see that
    
        \begin{align}
       \frac{\Cr{const:L2Grad_u_bound}(\overline{m})}{2}m^2\ge &\int_{\mathbb{S}^3}\csc^2(\theta)|\nabla u|dV_{g}-8\pi\geq\int_{E_{\tau,g,t}}\csc^2(\theta)|\nabla u|dV_{g}-8\pi.
        \end{align}
    Adding and subtracting $a(g)\int_{E_{\tau,g,t}}\csc(\theta)dV_{g}$ to the above gives
        \begin{align}
        \begin{split}
            \frac{\Cr{const:L2Grad_u_bound}(\overline{m})}{2}m^2&\ge\int_{E_{\tau,g,t}}\csc(\theta)\left(\frac{|\nabla u|}{\sin\theta}-a(g)\right)dV_{g}
            \\&+a(g)\int_{E_{\tau,g,t}}\csc(\theta )dV_{g}-8\pi.
        \end{split}
        \end{align}
    Then, using the definition of $E_{\tau,g,t}$ and adding and subtracting $\int_{E_{\tau,g,t}}\csc(\theta)dV_{g}$, we have
        \begin{align}
        \begin{split}
            \frac{\Cr{const:L2Grad_u_bound}(\overline{m})m^2}{2}\geq&\left(\int_{E_{\tau,g,t}}\csc(\theta )dV_g-8\pi\right)-\tau\int_{E_{\tau,g,t}}\csc(\theta) dV_{g}\\
            &-\left(1-a(g)\right)\int_{E_{\tau,g,t}}\csc(\theta)dV_{g}.
        \end{split}
        \end{align}
    Rewriting $8 \pi = \int_{\Sp^3}\csc(\theta)dV_{g_{\Sp^3}}=\int_{E_{\tau,g,t}}\csc(\theta)dV_{g_{\Sp^3}}+\int_{E^c_{\tau,g,t}}\csc(\theta)dV_{g_{\Sp^3}}$,  the last inequality becomes 
        \begin{align}
        \begin{split}
            \frac{\Cr{const:L2Grad_u_bound}(\overline{m})m^2}{2} {}&\ge\int_{E_{\tau,g,t}}\csc(\theta)\left(dV_{g}-dV_{\mathbb{S}^3}\right)-\int_{E_{\tau,g,t}^c}\csc(\theta) dV_{g_{\mathbb{S}^3}}\\
            {}&\quad-\tau\int_{E_{\tau,g,t}}\csc(\theta) dV_{g}-\left(1-a(g)\right)\int_{E_{\tau,g,t}}\csc(\theta)dV_{g}
            \end{split}\\
            \begin{split}{}&\geq|E_{\tau,g,t}|_g-|E_{\tau,g,t}|_{g_{\Sp^3}}-\int_{E^c_{\tau,g,t}}\csc(\theta)dV_{g_{\Sp^3}}-\tau\int_{E_{\tau,g,t}}\csc(\theta) dV_{g}\\
            {}&-\left(\Cr{const:effective_a_lower_bound_1}(\overline{m})m^{\frac{1}{12}}+\Cr{const:effective_a_lower_bound_2}(\overline{m},\Lambda)m^{\frac14}\right)\int_{E_{\tau,g,t}}\csc(\theta)dV_{g}
            \end{split}\label{2ndToLastEq}   
        \end{align}
    where we have used $\csc(\theta)\ge 1$, $dV_g \ge dV_{\mathbb{S}^3}$, and applied the estimate of $a(g)$ from Lemma \ref{lem:u_amplitude_lower_bound}.
    Now using the fact that $\csc(\theta) \le \csc(t)$ on $E_{\tau,g,t}$, $|\Sp^3|_{g}\leq V$, and applying H\"{o}lder's inequality to the integral over $E^c_{\tau,g,t}$ appearing in \eqref{2ndToLastEq}, we find
        \begin{align}    
        \begin{split}
            \frac{\Cr{const:L2Grad_u_bound}(\overline{m})m^2}{2}{}&\geq|E_{\tau,g,t}|_g-|E_{\tau,g,t}|_{g_{\Sp^3}}-\|\csc(\theta)\|_{L^2(g_{\Sp^3})}|E_{\tau,g,t}^c|_{g_{\Sp^3}}^{\frac{1}{2}}-\tau\csc (t)|\mathbb{S}^3|_g\\
            {}&\quad-\left(\Cr{const:effective_a_lower_bound_1}(\overline{m})m^{\frac{1}{12}}+\Cr{const:effective_a_lower_bound_2}(\overline{m},\Lambda)m^{\frac14}\right)\csc(t)V\label{2ndToLastEq2}. 
            \end{split} 
            \end{align}
   Finally, by \eqref{eq:volEtaut1} and \eqref{eq:tau_csct_bound} we find
        \begin{align}
        \begin{split}
            \frac{\Cr{const:L2Grad_u_bound}(\overline{m})m^2}{2}   {}&\geq|E_{\tau,g,t}|_g-|E_{\tau,g,t}|_{g_{\Sp^3}}\\
            {}&-\|\csc(\theta)\|_{L^{2}(g_{\Sp^3})}\left(\Cr{const:Ratio_grad_u-grad_sin-Poincare}(\overline{m},\Lambda)m^{\frac{1}{2}}+\frac{8\pi}{3}m^{\frac{1}{8}}\right)^{\frac{1}{2}}-\Cr{const:smallanglecsc}(\overline{m})Vm^{\frac{11}{24}}\\
            {}&-\left(\Cr{const:effective_a_lower_bound_1}(\overline{m})m^{\frac{1}{12}}+\Cr{const:effective_a_lower_bound_2}(\overline{m},\Lambda)m^{\frac14}\right)\Cr{const:smallanglecsc}(\overline{m})m^{-\frac{1}{24}}.
        \end{split}
        \end{align}
    The proof is complete upon noting that $\|\csc(\theta)\|_{L^{2}(g_{\Sp^3})}$ is finite, $m\leq 1$, and that the smallest power of $m$ appearing is $\tfrac{1}{48}$.
\end{proof}
    We now apply Lemma \ref{lem:vol_Etaut0} together with Theorem \ref{thm-VADB} to prove Theorem \ref{thm-Main}.

    \begin{proof}[Proof of Theorem \ref{thm-Main}]
        In light of Theorem \ref{thm-VADB}, it is sufficient to show
        \begin{equation}\label{eq-mainproof:volume}
            \lim_{i\rightarrow\infty}\abs{\Sp^3}_{g_{i}}=\abs{\Sp^3}_{g_{\Sp^3}}.
        \end{equation}
    Equip each $(\Sp^3,g_i)$ with a Llarull potential centered at some poles $\pm p_i$. Set $m_i=\left\|\left(6-R_{g_i}\right)^{+}\right\|_{L^{2}(g_i)}^{1/2}$, note that we may assume $m_i\leq1$, and choose $\tau_i=m_i^{1/2}$ and $t_i=m_i^{1/24}$ so that Lemma \ref{lem:vol_Etaut0} reads
        \begin{align} 
        \begin{split}
            |\Sp^3|_{g_i}-|E_{\tau_i,g_i,t_i}^c|_{g_i}&=|E_{\tau_i,g_i,t_i}|_{g_i}\\
            {}&\leq |E_{\tau_i,g_i,t_i}|_{g_{\Sp^3}}+\Cr{const:vol_Etaut0}(\overline{m},V,\Lambda) m_i^{1/24}\\
            {}&\leq |\Sp^3|_{g_{\Sp^3}}+\Cr{const:vol_Etaut0}(\overline{m},V,\Lambda) m_i^{1/24}.\label{eq:final1}
        \end{split}
        \end{align}
    By Lemma \ref{lem:small_theta_balls} applied with radii $r=t_i$, we are free to choose the poles $\pm p_i$ so that     
    \begin{align}
    \begin{split}
    |E_{\tau_i,g_i,t_i}^c|_{g_i}&\leq|E_{\tau_i,g_i}^c|_{g_i}+|B^{\Sp}(p_i,t_i)|_{g_i}+|B^{\Sp}(-p_i,t_i)|_{g_i}
    \\&\leq \Cr{const:Ratio_grad_u-grad_sin-Poincare}(\overline{m},\Lambda)m_i^{1/2} + \Cr{const:radius}(V) m_i^{1/8}.\label{eq:final2}
    \end{split}
    \end{align}
    Combining \eqref{eq:final1} and \eqref{eq:final2} yields 
    \begin{align}
    \begin{split}
     |\Sp^3|_{g_{\Sp^3}} &\leq   |\Sp^3|_{g_i}
     \\&\leq |\Sp^3|_{g_{\Sp^3}}+\Cr{const:vol_Etaut0}(\overline{m},V,\Lambda) m_i^{1/24}
      +\Cr{const:Ratio_grad_u-grad_sin-Poincare}(\overline{m},\Lambda)m_i^{1/2} + \Cr{const:radius}(V)m_i^{1/8}.       
    \end{split}
    \end{align}
    Taking the limit as $i\to\infty$ establishes \eqref{eq-mainproof:volume}.
    \end{proof}
    \bibliographystyle{plain}
    \bibliography{bibliography}

\begin{bibdiv}
\begin{biblist}

\bib{Allen-Perales-Sormani-VADB}{misc}{
      author={Allen, Brian},
      author={Perales, Raquel},
      author={Sormani, Christina},
       title={Volume above distance below},
   publisher={arXiv},
        date={2020},
         url={https://arxiv.org/abs/2003.01172},
        note={arXiv:2003.01172},
}

\bib{Dong-PMT_Stability}{misc}{
      author={Dong, Conghan},
       title={Stability of the positive mass theorem in dimension three},
   publisher={arXiv},
        date={2022},
         url={https://arxiv.org/abs/2211.06730},
        note={arXiv:2211.06730},
}

\bib{Dong-Song}{misc}{
      author={Dong, Conghan},
      author={Song, Antoine},
       title={Stability of euclidean 3-space for the positive mass theorem},
        date={2023},
        note={arXiv:2302.07414},
}

\bib{folland1999}{book}{
      author={Folland, Gerald~B},
       title={Real analysis: modern techniques and their applications},
   publisher={John Wiley \& Sons},
        date={1999},
      volume={40},
}

\bib{GL1}{article}{
      author={Gromov, Mikhael},
      author={Lawson, H.~Blaine, Jr.},
       title={Spin and scalar curvature in the presence of a fundamental group.
  {I}},
        date={1980},
        ISSN={0003-486X},
     journal={Ann. of Math. (2)},
      volume={111},
      number={2},
       pages={209\ndash 230},
         url={https://doi.org/10.2307/1971198},
      review={\MR{569070}},
}

\bib{GL2}{article}{
      author={Gromov, Mikhael},
      author={Lawson, H.~Blaine, Jr.},
       title={Positive scalar curvature and the {D}irac operator on complete
  {R}iemannian manifolds},
        date={1983},
        ISSN={0073-8301},
     journal={Inst. Hautes \'{E}tudes Sci. Publ. Math.},
      number={58},
       pages={83\ndash 196 (1984)},
         url={http://www.numdam.org/item?id=PMIHES_1983__58__83_0},
      review={\MR{720933}},
}

\bib{GromovBoundaries}{misc}{
      author={Gromov, Misha},
       title={Scalar curvature of manifolds with boundaries: Natural questions
  and artificial constructions},
        date={2018},
        note={arXiv:1811.04311},
}

\bib{GromovFour}{incollection}{
      author={Gromov, Mikhail~L.},
       title={Four lectures on scalar curvature},
        date={2023},
   booktitle={Chapter for perspectives in scalar curvature. {V}ol. {1}},
      editor={Gromov, Mikhail~L.},
      editor={Lawson~Jr., H.~Blaine},
   publisher={World Scientific},
}

\bib{hirsch-kazaras-khuri}{misc}{
      author={Hirsch, Sven},
      author={Kazaras, Demetre},
      author={Khuri, Marcus},
       title={Spacetime harmonic functions and the mass of 3-dimensional
  asymptotically flat initial data for the einstein equations},
        date={2021},
        note={arXiv:2002.01534},
}

\bib{Hirsch-Kazaras-Khuri-Zhang}{misc}{
      author={Hirsch, Sven},
      author={Kazaras, Demetre},
      author={Khuri, Marcus},
      author={Zhang, Yiyue},
       title={Rigid comparison geometry for {R}iemannian bands and open
  incomplete manifolds},
   publisher={arXiv},
        date={2022},
         url={https://arxiv.org/abs/2209.12857},
        note={arXiv:2209.12857},
}

\bib{Sweeney}{misc}{
      author={Jr., Paul~Sweeney},
       title={Examples for scalar sphere stability},
        date={2023},
        note={arXiv:2301.01292},
}

\bib{li_2012}{book}{
      author={Li, Peter},
       title={Geometric analysis},
      series={Cambridge Studies in Advanced Mathematics},
   publisher={Cambridge University Press},
        date={2012},
}

\bib{Llarull}{article}{
      author={Llarull, Marcelo},
       title={Sharp estimates and the {D}irac operator},
        date={1998},
     journal={Mathematische Annalen},
       pages={55\ndash 71},
}

\bib{SormaniIAS}{incollection}{
      author={Sormani, Christina},
       title={Conjectures on convergence and scalar curvature},
        date={2023},
   booktitle={Chapter for perspectives in scalar curvature. {V}ol. {2}},
      editor={Gromov, Mikhail~L.},
      editor={Lawson~Jr., H.~Blaine},
   publisher={World Scientific},
}

\bib{Stern}{article}{
      author={Stern, Daniel~L.},
       title={Scalar curvature and harmonic maps to {$S^1$}},
        date={2022},
        ISSN={0022-040X,1945-743X},
     journal={J. Differential Geom.},
      volume={122},
      number={2},
       pages={259\ndash 269},
         url={https://doi.org/10.4310/jdg/1669998185},
      review={\MR{4516941}},
}

\bib{Sormani-Wenger}{article}{
      author={Sormani, Christina},
      author={Wenger, Stefan},
       title={The intrinsic flat distance between {R}iemannian manifolds and
  other integral current spaces},
        date={2011},
        ISSN={0022-040X},
     journal={J. Differential Geom.},
      volume={87},
      number={1},
       pages={117\ndash 199},
  url={http://projecteuclid.org.lehman.ezproxy.cuny.edu/euclid.jdg/1303219774},
      review={\MR{2786592}},
}

\bib{Schoen-Yau-Structure}{article}{
      author={Schoen, R.},
      author={Yau, S.~T.},
       title={On the structure of manifolds with positive scalar curvature},
        date={1979},
     journal={Manuscripta Mathematica},
      volume={28},
      number={1–3},
       pages={159–183},
}

\end{biblist}
\end{bibdiv}
\end{document}